\documentclass[12pt]{article}
\usepackage[margin=0.8in,footskip=0.25in]{geometry}
\usepackage[utf8]{inputenc}

\usepackage{amsmath,amssymb,amsthm,color,mathtools}
\usepackage{caption}

\usepackage{siunitx}

\usepackage[numbers]{natbib}
\makeatletter
\def\NAT@spacechar{~}
\makeatother

\theoremstyle{plain}

\newtheorem{lem}{Lemma}

\newtheorem{Proposition}{Proposition}

\newtheorem{Theorem}{Theorem}

\theoremstyle{remark}
\newtheorem*{rmrk*}{Remark}
\newtheorem{rmrk}{Remark}

\theoremstyle{definition}
\newtheorem{ex}{Example}

\newtheorem{assumption}{Assumption}

\usepackage[textsize=scriptsize]{todonotes}

\usepackage{tcolorbox}

\numberwithin{equation}{section}
\usepackage{enumerate}
\usepackage{graphicx}
\usepackage{color}

\usepackage[pdftex,colorlinks = true]{hyperref}
\hypersetup{citecolor=darkblue,linkcolor=darkblue,urlcolor=darkblue}
\definecolor{darkblue}{rgb}{0,0,.6}

\newcommand{\ii}{\mathrm{i}}

\newcommand{\vertiii}[1]{{\left\vert\kern-0.25ex\left\vert\kern-0.25ex\left\vert #1 
    \right\vert\kern-0.25ex\right\vert\kern-0.25ex\right\vert}}

\newcommand{\convP}{\stackrel{P}{\to}}
\newcommand{\convd}{\stackrel{\mathrm{d}}{\to}}

\newcommand{\D}{S_n^V}
\newcommand{\Dsub}[1]{S_{n,#1}^V}

\usepackage{authblk}

\title{The maximum of the periodogram of Hilbert space valued time series}
\author[a]{Cl\'ement Cerovecki}
\author[b]{Vaidotas Characiejus}
\author[c]{Siegfried H\"ormann}
\affil[a]{D\'epartement de math\'ematique, Universit\'e libre de Bruxelles, Belgium}
\affil[a]{Department of Mathematics, Katholieke Universiteit Leuven, Belgium}
\affil[b]{Department of Statistics, University of California, Davis, USA}
\affil[c]{Institute of Statistics, Graz University of Technology, Graz, Austria}
\date{July 4, 2020}

\begin{document}

\maketitle

\begin{abstract}
\noindent We are interested to detect periodic signals in Hilbert space valued time series when the length of the period is unknown. A natural test statistic is the maximum Hilbert-Schmidt norm of the periodogram operator over all fundamental frequencies. In this paper we analyze the asymptotic distribution of this test statistic. We consider the case where the noise variables are independent and then generalize our results to functional linear processes. Details for implementing the test are provided for the class of functional autoregressive processes.  We illustrate the usefulness of our approach by examining air quality data from Graz, Austria. The accuracy of the asymptotic theory in finite samples is evaluated in a simulation experiment.

\medskip
\noindent \textbf{Keywords:} periodogram, periodicities, spectral analysis, time series, functional data, hypothesis testing.

\medskip
\noindent \textbf{MSC2020:} 62G32, 62M10, 62M15, 62R10, 62G10.
\end{abstract}

\section{Introduction}\label{introunif}
Periodic characteristics are present in many time series due to various factors such as different seasons, meteorological phenomena, human economic activity, transport, etc. The interest to detect, analyze and model periodicities goes back to the origins of time series analysis (for example, \citet{schuster:1898}, \citet{walker:1914}, \citet{yule:1927}, \citet{fisher:1929}, \citet{grenander:rosenblatt:1957},  \citet{jenkins:priestley:1957}, \citet{hannan:1961}, \citet{shimshoni:1971} to name just a few).

The primary motivation of this paper is to develop a methodology to detect periodicities in functional time series (FTS). This a sequence $\{X_t\}_{t\geq 1}$, where each $X_t$ is a curve $\{X_t(u)\}_{u\in \mathcal{U}}$. FTS have been gaining interest in recent years due to the advances of modern technology and the availability of high frequency data. Frequently, FTS arise from measurements obtained by separating a continuous time process $\{Y(u)\}_{u\geq 0}$ into natural consecutive intervals, for instance, days. Then, in an appropriate time scale we have $X_t(u)=Y(t+u)$ for $u\in \mathcal{U}=[0,1]$. Examples include
volume of credit card transactions (\citet{laukaitis:2002}),
electricity spot prices (\citet{liebl:2013}),
high frequency asset price data (\citet{horvath:kokoszka:rice:2014}), 
daily pollution level curves (\citet{aue:dubartNorinho:hormann:2015}), 
daily vehicle traffic curves (\citet{klepsch:kluppelberg:wei:2016}), etc. It should be noted, that such a segmentation already accounts for a periodic structure in the underlying continuous time process.  For example, when we segment into daily data, it is because we expect a similar daily fluctuation in each curve. Our interest is then to investigate if there remains a periodic behavior with respect to the discrete time parameter $t$.

While this problem is well explored in the univariate setting (see Section~10.2 of \citet{brockwell:davis:1991} for an overview of the classical tests), developments in multivariate or functional context are restricted to periodicity tests where the length of the period is known (see \citet*{macneill:1974} and \citet*{hoermann:2018}).
This paper is motivated by the interest in \emph{testing for an unspecified period}, which makes the problem considerably more complex and requires an entirely different theoretical approach. Testing for an unspecified period (in residuals or raw data) is relevant, because periodic behavior can have diverse causes and quite often is not evident. Even though sometimes we expect that the data contains, for example, a weekly or monthly periodic component, there are situations when the period of a latent signal is not so evident. For instance, the solar cycle is a nearly periodic $11$-year change in the Sun's activity measured in terms of variations in the number of observed sunspots on the solar surface discovered by \citet{schwabe:1843}. In Section~\ref{test} we also show that our test indicates an unexpected periodic component in the air quality data set from Graz, Austria.

Our test is based on the frequency domain approach to FTS analysis, which is rather natural in this context. This topic has been gaining a significant amount of attention in recent years and it is very useful in various problems (see, for example, \citet{panaretos:tavakoli:2013a}, \citet*{hoermann:kidzinski:hallin:2015}, \citet{zhang:2016}, \citet{characiejus:rice:2020} among others). For the theoretical developments which follow we consider time series with values in an abstract separable Hilbert space. In this way we cover functional and multivariate data. For the latter our results are also new.

Before we describe our approach in detail, we introduce notation that is used throughout the paper. Suppose that $H_0$ is a real separable Hilbert space equipped with an inner product~$\langle \cdot,\cdot\rangle:H_0\times H_0\to\mathbb R$ and the corresponding norm~$\|\cdot\|:H_0\to[0,\infty)$. The complexification of $H_0$ is denoted by $H\coloneqq H_0\oplus \ii H_0$ and the space $H$ inherits the Hilbert space structure from $H_0$. The complex inner product is defined as $\langle u,v\rangle_{H_0} = \langle u_0,v_0 \rangle + \langle u_1,v_1 \rangle +\ii(\langle u_1,v_0 \rangle - \langle u_0,v_1 \rangle)$ for any $u = u_0+\ii u_1$ and $v = v_0+\ii v_1$ in $H$ with $u_0,u_1,v_0,v_1\in H_0$. For easier notation, we henceforth consider $H_0$ as a subspace of $H$ and use $\langle \cdot,\cdot\rangle$ for the real and the complex inner product. We do the same for the norm and other definitions to come. $\mathcal{L}(H)$ denotes the space of bounded linear operators on $H$ and it is equipped with the usual operator norm $\|A \|=\sup_{\|x\|=1}\|A(x)\|$. We say that an operator $A$ is Hilbert-Schmidt (trace-class) if its singular values $\{\sigma_k\}_{k\geq 1}$ are square summable (absolutely summable). We define the corresponding Hilbert-Schmidt norm $\|A\|_{\mathcal S}=(\sum_{k=1}^\infty\sigma_k^2)^{1/2}$ and the trace norm $\|A\|_{\mathcal{T}}=\sum_{k=1}^\infty\sigma_k$ (see \citet{weidmann:1980} for more details). For $x,y\in H$, the tensor of $x$ and $y$ is a rank one operator $x\otimes y:H\to H$ defined by $(x\otimes y)(z)\coloneqq\langle z,y\rangle x$ for each $z\in H$. In particular this gives rise to the covariance operator $\mathrm{Var}(X) := E[(X-EX)\otimes (X-EX)]$. We note that for $H_0=\mathbb{R}^d$ with $d>1$ this is the usual covariance matrix. For more details on random elements in Hilbert spaces, we refer to \citet{bosq:2000}.

Suppose that $X_1,\ldots,X_n$ are observations of random elements with values in some separable Hilbert space $H_0$ and define the discrete Fourier transform (DFT) of $X_1,\ldots,X_n$ by setting
\[
	\mathcal X_n(\omega)\coloneqq\frac1{\sqrt{n}}\sum_{t=1}^n X_t e^{-\ii t\omega}
\]
for $n\ge1$ and $\omega\in[-\pi,\pi,]$, where $\ii=\sqrt{-1}$. By its very definition $\mathcal{X}_n(\omega)$ is an element of the complex Hilbert space $H:=H_0\oplus \ii H_0$ for $\omega\in[-\pi,\pi]$. The periodogram operator is defined by
 \begin{equation}\label{per}
	I_n(\omega)
	\coloneqq\mathcal X_n(\omega)\otimes \mathcal X_n(\omega)
\end{equation}
for $n\ge1$ and $\omega\in[-\omega,\omega]$ and it is a well-known and important tool in time series analysis. It is the main ingredient for estimation of the spectral density operator (see \citet{panaretos:tavakoli:2013a}) and it is the key statistic for detection of periodic signals in the data. What is of particular interest is the maximum of the periodogram operator defined by
\begin{equation} \label{MnX}
	M_n
	\coloneqq\max_{j=1,\dots,q}\|I_n(\omega_j)\|_{\mathcal S}
	=\max_{j=1,\dots,q}\| \mathcal{X}_n(\omega_j)\|^2,
\end{equation}
where $\omega_1,\ldots,\omega_q$ are the Fourier or the fundamental frequencies given by $\omega_j=2\pi j/n$ with $j=1,\ldots,q$ and $q=\lfloor (n-1)/2\rfloor$. In the univariate case the exact distribution of $M_n$ can be derived for independent and identically distributed (iid) Gaussian data (see \citet{fisher:1929} as well as Section~10.2 of \citet{brockwell:davis:1991}). Then $M_n$ is the maximum of $q$ iid standard exponential random variables and $M_n$ belongs to the domain of attraction of the standard Gumbel distribution. That is, $M_n-\log q$ converges in distribution to the standard Gumbel distribution as $n\to\infty$ (the cumulative distribution function of the standard Gumbel distribution is given by $F(x)=\exp\{-e^{-x}\}$ for $x\in\mathbb R$). If we superimpose a sinusoidal signal $s_t=\alpha\cos(\theta+\omega_j t)$ for some $j\in \{1,\ldots, q\}$ to the observations, then $M_n$ will diverge at a rate proportional to $n$, which in turn then leads to a very powerful test statistic. 

The assumption of Gaussianity is restrictive and hence an alternative approach would be to establish the asymptotic distribution of the appropriately standardized maximum $M_n$ under more general conditions.
\citet{walker:1965} conjectured that the same result still holds even if the random variables are not normal, provided that the moments of the distribution of $X_1,\ldots,X_n$ up to some sufficiently high order exist. \citet{walker:1965} also stated that no proof was known at the time and that the problem of constructing one is undoubtedly extremely difficult. Almost $35$ years later, \citet{davis:mikosch:1999} proved that the limit indeed remains the same provided that $\operatorname E|X_1|^s<\infty$ with some $s>2$ using a Gaussian approximation technique due to \citet{einmahl:1989}. Later on the results of \citet{davis:mikosch:1999} were extended by~\citet{lin:2009} to a broad class of stationary processes.

The main result of this paper is an extension of the result of \citet{davis:mikosch:1999} (see Theorem~2.1 therein) to real separable Hilbert spaces (finite dimensional or infinite dimensional) under certain technical conditions (see Theorems~\ref{mainthm0}, \ref{thm:finite}, \ref{mainthm}, and \ref{fancythm}). A key ingredient of our proof is a powerful Gaussian approximation developed by~\citet{chernozhukov:2017} which in turn relies on an anti-concentration inequality due to~\citet{nazarov:2003}. In fact, we obtain a slight extension of Proposition~3.2 of \citet{chernozhukov:2017} by making the dependence of the bound on certain parameters explicit (see \autoref{p:cher+c} as well as Appendix~\ref{bdep}). This result might be of independent interest. 

In many situations, assuming that the observations are iid random elements is not realistic and hence we provide extensions of our main results to dependent sequences. Following the classical approach of~\citet{walker:1965}, we provide a generalization to linear processes (see \autoref{cor:lpinfty} as well as \autoref{thm:mv}). Our \autoref{lem:lp} not only extends the results of \citet{walker:1965} but also does so under weaker conditions.

These results allow us to construct tests for hidden periodicities in time series with values in a separable Hilbert space which complement the methods of \citet{hoermann:2018}, where the length of the period is assumed to be known. Specifically, we want to test the null hypothesis $\mathcal H_0$ that the observations are generated by a linear process (no periodic component) against the alternative hypothesis $\mathcal H_1$ that the observations are generated by a linear process with a superimposed deterministic periodic component with an unknown period.  We also establish the consistency of the proposed test (see \autoref{th:applconsist}) without assuming any specific shape or form of the superimposed deterministic periodic component.

The rest of the paper is organized as follows. In Section~\ref{mainresultssec} we formulate our main theorems which are valid for iid data. In Section~\ref{s:lin} we extend these results to linear processes. Then we illustrate in Section~\ref{test} how to use our results to construct a test for periodic signals in functional time series at some unknown frequency. We evaluate the finite sample behavior in a simulation study and with a real data example in Section~\ref{s:empirical}. We give a conclusion in Section~\ref{s:conclusion} and provide the proofs in Section~\ref{proofsart3}. In the Appendix, we prove two theorems which are of separate interest and which are needed for proving our main results.

\section{Main results} \label{mainresultssec}
Suppose that $\{X_t\}_{t\ge1}$ are iid random elements with values in $H_0$ such that $ EX_1=0$ and $E\|X_1\|^2<\infty$. Let $\{v_k\}_{k\ge1}$ be the eigenvectors (principal components) of the covariance operator $E[X_1\otimes X_1]$ with their corresponding eigenvalues $\{\lambda_k\}_{k\ge1}$. The $\{v_k\}_{k\ge1}$ form an orthonormal basis of $H_0$ and  $\{\lambda_k\}_{k\ge1}$ are indexed in a non-increasing order. We use the following assumption in some of our results.
\begin{assumption}\label{ass:eigen}
$\lambda_k>\lambda_{k+1}$ for each $k\ge1$.
\end{assumption}
\noindent
Below we use $V\sim \operatorname{Exp}(\theta)$ to indicate that $V$ follows an exponential distribution with mean $1/\theta$ and $V\sim \operatorname{Hypo}(\theta_1,\ldots, \theta_p)$ if the variable $V$ follows a hypoexponential distribution, i.e.\ $V\sim \sum_{k=1}^pE_i$ where $E_i$ are independent $\operatorname{Exp}(\theta_i)$ random variables with $1\le i\le p$. As usual, $N(\mu,\sigma^2)$ denotes the normal distribution with mean $\mu$ and variance $\sigma^2$.
\subsection{The multivariate setup}
We start by studying the projections of $X_t$'s onto the space spanned by $\{v_1,\ldots, v_d\}$:
\[
	X_t^d=\sum_{k=1}^d\langle X_t,v_k\rangle v_k,\quad t\geq 1.
\]
The DFT and the periodogram operator of $\{X_t^d\}_{1\le t\le n}$ are defined by
\begin{equation}\label{MnXp}
	\mathcal X_n^d(\omega)
	=n^{-1/2}\sum_{t=1}^nX_t^de^{-\ii t\omega}
	\quad \text{and}\quad
	I_n^d(\omega)
	=\mathcal{X}_n^d(\omega) \otimes \mathcal{X}_n^d(\omega),
\end{equation}
respectively, for $\omega\in[-\pi,\pi]$. Observe that  $X_t^d=X_t$ and  $\mathcal{X}_n^d(\omega)=\mathcal{X}_n(\omega)$ if  $H_0= \mathbb{R}^d$. So the multivariate setting can be viewed as a special case.
 
If we assume for the moment that the $X_t$'s are iid Gaussian random elements, then we have that
\begin{equation}\label{eq:intuition}
        \max_{1\leq j\leq q}  \|\mathcal{X}_n^d(\omega_j)\|^2
	=\max_{1\le j\le q}\Bigl\{\sum_{k=1}^d\lambda_kE_{kj}\Bigr\},
\end{equation}
where $E_{kj}$ are independent $\operatorname{Exp}(1)$ random variables for $1\le k\le d$ and $1\le j\le q$. This follows from the orthogonality of $\{v_k\}_{k\ge1}$, which implies that $\langle X_t,v_k\rangle$ are independent $N(0,\lambda_k)$ random variables. Consequently, $\|\mathcal{X}_n^d(\omega_j)\|^2$ are independent $\operatorname{Hypo}(\lambda_1^{-1},\ldots,\lambda_d^{-1})$ random variables. To have a non-degenerate limiting distribution, the variable $M_n$ needs to be centered and scaled. The corresponding sequences depend on the eigenvalues of $\operatorname{Var}(X_1^d)$.  If \autoref{ass:eigen} holds, then we have that $\lambda_1^{-1}(\max_{1\leq j\leq q}  \|\mathcal{X}_n^d(\omega_j)\|^2-b_q^d)\convd\mathcal{G}$ as $n\to\infty$, where $\mathcal{G}$ denotes a standard Gumbel distribution and where 
\begin{equation}\label{defseq}
b_{n}^d=\lambda_1\log(n\alpha_{1,d})\quad\text{and}\quad\alpha_{1,d} =  \prod_{j=2}^d(1-\lambda_j/\lambda_1)^{-1}
\end{equation}
(see \autoref{hypofixeddim} in Section~\ref{sec:Gum}).

If $H_0=\mathbb R^d$, and if $\Sigma\coloneqq\operatorname E[X_tX_t']$ has full rank we consider the standardized process $\{\Sigma^{-1/2}X_t\}_{t\ge1}$. Alternatively, we may directly assume that $\mathrm{Var}(X_1)=I_d$, where $I_d$ is the identity matrix. Then $\max_{1\leq j\leq q}  \|\mathcal{X}_n(\omega_j)\|^2-c_q\xrightarrow d\mathcal G$ as $n\to\infty$, where
\begin{equation}\label{eq:stand:c_n}
	c_n=\log n+(d-1)\log\log n-\log(d-1)!
\end{equation}
for $n\ge3$ (see Example~1 of \citet{kang:serfozo:1999} or Table~3.4.4 of \citet{embrechts:1997}).

In the following two theorems, we extend these results to iid random elements provided that the moments up to some sufficiently high order exist.
\begin{Theorem}\label{mainthm0}
Let $\{X_t\}_{t\geq 1}$ be iid random elements in $H_0$ with $E\|X_1\|^r<\infty$ for some $r>2$. Suppose that \autoref{ass:eigen} holds and $d\ge1$ is fixed. Then
\begin{equation}\label{gumbconv}
\lambda_1^{-1}(\max_{1\leq j\leq q}  \|\mathcal{X}_n^d(\omega_j)\|^2 - b_{q}^d) \convd \mathcal{G}\quad\text{as}\quad n\to\infty,
\end{equation}
where $b_q^d$ is given by \eqref{defseq}.
\end{Theorem}

\begin{Theorem}\label{thm:finite}
Let $\{X_t\}_{t\geq 1}$ be iid random vectors in $\mathbb{R}^d$ with $E\|X_1\|^r<\infty$ for some $r>2$ and $ E[X_1X_1']=I_d$, where $I_d$ is the identity matrix. Then
\[
	\max_{1\leq j\leq q}  \|\mathcal{X}_n(\omega_j)\|^2-c_q\xrightarrow d\mathcal G\quad\text{as}\quad n\to\infty,
\]
where $c_q$ is given by \eqref{eq:stand:c_n}.
\end{Theorem}
The proofs of \autoref{mainthm0} and \autoref{thm:finite} are given in Section~\ref{proofsart3}. They rely on a powerful Gaussian approximation due to~\citet{chernozhukov:2017} (see \autoref{p:cher}). If $H_0=\mathbb R$, we recover Theorem~2.1 of \citet{davis:mikosch:1999} as a special case of \autoref{thm:finite}. To the best of our knowledge, \autoref{mainthm0} is the first multivariate generalization of Theorem 2.1 of~\citet{davis:mikosch:1999}.

We now present an extension of \autoref{mainthm0}, where we let $d$ grow to infinity as $n\to\infty$.
\begin{Theorem}\label{mainthm}
Suppose that $E\|X_1\|^4<\infty$ and that \autoref{ass:eigen} holds. Assume that $\{k\lambda_k\}_{k\ge1}$ is eventually monotonic, i.e.\ there exists $k_0\ge1$ such that
\begin{equation} \label{assLem5}
k\lambda_k\ge (k+1)\lambda_{k+1}
\end{equation}
for all $k\ge k_0$. Then convergence \eqref{gumbconv} still holds if $d$ is replaced by a sequence of integers $\{d_n\}_{n\geq 1}$ such that $d_n\to\infty$, and
\begin{equation} \label{assmain}
  \frac{ d_n^{4}  }{\lambda_{d_n}^{1/2} }=o(n^{1/6}/\log^{7/6}n) \qquad \text{and} \qquad d_n= O(n^{\gamma_0})
\end{equation}
as $n\to\infty$ with 
\begin{equation}\label{gamma0}
\gamma_0 \; <\; \min\Bigl\{\min_{k\geq 2}\Bigl\{ \frac{1}{k} \Bigl(\frac{\lambda_1}{\lambda_k} -1\Bigr)\Bigr\},1\Bigr\}.
\end{equation}
\end{Theorem}
Since we assume that $\lambda_k$'s are strictly decreasing and summable, we have that $k\lambda_k=o(1)$ as $k\to\infty$ and hence we only additionally require $\{k\lambda_k\}_{k\ge1}$ to be eventually monotonic in \autoref{mainthm}.
The first condition in \eqref{assmain} ensures that the Gaussian approximation still holds while the second condition in \eqref{assmain}, as well as \eqref{assLem5} and \eqref{gamma0} are used to show that the hypoexponential distribution with an increasing number of parameters belongs to the domain of attraction of the Gumbel distribution (see \autoref{lemalpha} in Section~\ref{sec:Gum}).
If $d_n\to\infty$ as $n\to\infty$, we can choose a centring sequence $\{b_n\}_{n\geq 1}$ independently of $\{d_n\}_{n\ge1}$ by setting $b_n=\lim_{d\to\infty}b_n^d$ for $n\ge1$, where $b_n^d$ is defined by \eqref{defseq} (see \autoref{munp}).

\subsection{The infinite dimensional case}
The following theorem establishes a fully functional result, i.e.\ the convergence in distribution of  $\lambda_1^{-1}(M_n-b_q)$ as $n\to\infty$, where $M_n$ is defined by \eqref{MnX}. Technical conditions are connected with the decay rate of the eigenvalues $\{\lambda_k\}_{k\geq 1}$ of the covariance operator $\operatorname{Var}(X_1)$.
\begin{Theorem}\label{fancythm} Suppose that $E\|X_1\|^r<\infty$ for some $r\geq 4$ and let \autoref{ass:eigen} hold. Moreover, suppose that there exists a sequence $\{d_n\}_{n\geq 1}$ which satisfies the conditions of \autoref{mainthm}. Consider some sequence $\{\ell_k\}_{k\geq 1}$ of positive numbers such that $ \sum_{k=1}^\infty\ell_k = 1$ and assume that 
\begin{equation}\label{uglygood}
	\sum_{k=1}^\infty\ell_k^{-r/2}E|\langle X_1,v_k\rangle|^r<\infty
 \end{equation}
 and that
 \begin{equation}\label{condlemmMMb} 
 \sum_{k>d_n}   (\lambda_k/\ell_k)^{r/2}  = o(1/n)
\end{equation}
as $n\to\infty$. Then $\lambda_1^{-1}(M_n- b_q) \convd \mathcal{G}$ as $n\to\infty$, where $b_q=\lim_{d\to\infty}b_q^d$ with $b_q^d$ given by \eqref{defseq}.
\end{Theorem}

\begin{rmrk}
By our assumption $r/2-2\ge0$, and hence \eqref{uglygood} implies
\begin{equation}\label{uglygoodreal}
	\sum_{k>d_n} \ell_k^{-r/2} E|\langle X_1,v_k\rangle|^r =  o(n^{r/2-2}).
\end{equation}
We prove \autoref{fancythm} under weaker condition \eqref{uglygoodreal}.
\end{rmrk}

If $X_1$ is a Gaussian random element, then $E|  \langle X_1,v_k\rangle|^r= E|Z|^r\cdot\lambda_k^{r/2}$, where $Z\sim N(0,1)$ and hence condition \eqref{condlemmMMb} implies condition \eqref{uglygood}. While under Gaussianity such an equality holds for any $r>0$, we only need this condition for some fixed $r\geq 4$. To this end, we note that by the Karhunen-Lo\`eve expansion any random element $X_1$ in $H_0$ has the representation
\[
	X_1
	=\sum_{k\geq 1}\langle X_1,v_k\rangle v_k
	=\sum_{k\geq 1} \sqrt{\lambda_k}Z_kv_k,
\]
where $\{Z_k\}_{k\geq 1}$ is white noise with mean zero and unit variance. Since $\langle X_1,v_k\rangle= \sqrt{\lambda_k} Z_k$, the condition 
\begin{equation}\label{simpcond}
\sup_{k\geq 1}E|Z_k|^r=C<\infty,
\end{equation} 
provides the bound $E|\langle X_1,v_k\rangle|^r\leq C\lambda_k^{r/2}$ for all $k\geq 1$. Consequently, \eqref{condlemmMMb} together with \eqref{simpcond} imply \eqref{uglygood}.

Let us provide two examples where the conditions of \autoref{fancythm} are satisfied. We look at the settings where the  eigenvalues $\lambda_k$ decay exponentially or polynomially. For numerical sequences $\{\alpha_n\}_{n\ge1}$ and $\{\beta_n\}_{n\ge1}$ we write $\alpha_n=\Theta(\beta_n)$ as $n\to\infty$ if there exist $k>0$, $K>0$  and $N\ge1$ such that $k\beta_n\le \alpha_n\le K\beta_n$ for all $n>N$.
\begin{ex}
Suppose that $\operatorname E\|X_1\|^r<\infty$ with $r>6$ and $\lambda_k=\Theta(\rho^k)$ with $0<\rho<1$ as $k\to\infty$. Also, assume that \eqref{assLem5} as well as \eqref{simpcond} hold. We choose $d_n = \lfloor c\log(n)\rfloor$ with
\[
            \frac{2}{r\log(1/\rho)}\,<\,c \, < \, \frac1{3\log(1/\rho)}.
\]
Then
\[
	\frac{ d_n^4  }{\lambda_{d_n}^{1/2} } 
	=O(\log^4(n)n^{\frac{c}{2}\log(1/\rho)})
	=o\Bigl(\,\frac{n^{1/6}}{\log^{7/6}n}\,\Bigr)
\] 
as $n\to\infty $ if $c<(3\log(1/\rho))^{-1}$. This shows that \eqref{assmain} holds. We set $\ell_k=\epsilon(1-\epsilon)^{-1}(1-\epsilon)^k$ for some $\epsilon\in (0,1-\rho)$. Then \eqref{condlemmMMb} holds since
\[
	\sum_{k>d_n}(\lambda_k/\ell_k)^{r/2}
	=O((\rho/(1-\epsilon))^{rd_n/2})
	=O(n^{-\frac{rc}{2}\log((1-\epsilon)/\rho)})
	=o(1/n)
\]
as $n\to\infty$ whenever $c> 2/(r\log(1/\rho))$ and if $\epsilon$ is small enough. Hence the required conditions hold.
\end{ex}
\begin{ex}
Suppose that $\lambda_k=\Theta(k^{-\nu})$ with $\nu>1$ as $k\to\infty$.  Now choose some large enough $r>2/(\nu-1)$
 such that for some $\beta>0$ 
\begin{equation}\label{eq:beta}
	\frac{1}{(\nu-1)r/2 -1}
	<\beta<
	\min\Bigl\{\frac{1}{3(8+\nu)},\min_{k\geq 2}\frac1k\Bigl(\frac{\lambda_1}{\lambda_k}-1\Bigr),1\Bigr\}
\end{equation}
and assume that $\operatorname E\|X_1\|^r<\infty$.  Also, let us assume that \eqref{assLem5} as well as \eqref{simpcond} hold. Then we may set $d_n=\lfloor n^{\beta}\rfloor $ and verify condition~\eqref{assmain} so that \autoref{fancythm} is applicable. To this end we notice that
\[
	\frac{ d_n^4  }{\lambda_{d_n}^{1/2} }
	=O(n^{\beta(4+\nu/2)})
	=o\Bigl(\frac{n^{1/6}}{\log^{7/6}n}\Bigr)
\] 
as $n\to\infty$ since $\beta<(3(8+\nu))^{-1}$. For the second part of condition \eqref{assmain}, we require $\beta<\min\{\min_{k\geq 2}k^{-1}(\lambda_1/\lambda_k -1),1\}$. 

In order to verify \eqref{condlemmMMb} we choose $\ell_k$ proportional to $k^{-(1+\epsilon)}$. Then
\[
	\sum_{k>d_n}  (\lambda_k/\ell_k)^{r/2}
	=O\Bigl(\sum_{k>d_n} k^{\frac{r}{2}(-\nu+1+\epsilon)}\Bigr)
	=O(n^{\beta(\frac{r}{2}(1+\epsilon-\nu)+1)})
	=o(n^{-1})
\]
as $n\to\infty$ if $r>2/(\nu-1)$ and if $\beta>1/((\nu-1)r/2-1)$, provided that $\epsilon$ is chosen small enough. 
\end{ex}

\section{Extension to linear processes}\label{s:lin}
We consider an extension of our \autoref{thm:finite} and \autoref{fancythm} to linear processes. Suppose that $\{X_t\}_{t\in\mathbb Z}$ is a linear process given by
\begin{equation}\label{eq:linproc}
	X_t
	=\sum_{k=-\infty}^\infty a_k(\varepsilon_{t-k})
\end{equation}
for each $t\in\mathbb Z$, where $\{a_k\}_{k\in\mathbb Z}\subset \mathcal{L}(H_0)$ such that $\sum_{k=-\infty}^\infty\|a_k\|<\infty$ and $\{\varepsilon_t\}_{t\in\mathbb Z}$ are iid  $H_0$-valued random elements with zero means. We denote the DFT of $\varepsilon_1,\ldots,\varepsilon_n$ by
\[
	\mathcal E_n(\omega)
	=n^{-1/2}\sum_{t=1}^n\varepsilon_te^{-\ii t\omega}
\]
for $\omega\in[-\pi,\pi]$ and $n\ge1$. We also use the impulse-response operator $A(\omega)$ defined by
\begin{equation}\label{eq:A(omega))}
	A(\omega)
	=\sum_{k=-\infty}^\infty a_ke^{-\ii k\omega}
\end{equation}
 for $\omega\in[\pi,\pi]$.

The next lemma establishes a relationship between the DFT and the periodogram operator of $X_1,\ldots,X_n$ and the DFT and the periodogram operator of $\varepsilon_1,\ldots,\varepsilon_n$. Essentially, this is a generalization of Theorem~3 of \citet{walker:1965} to linear processes with values in separable Hilbert spaces.
\begin{lem}\label{lem:lp}
Suppose that $\{X_t\}_{t\in\mathbb Z}$ is given by \eqref{eq:linproc} and $\sum_{k\ne 0}\log(|k|)\|a_k\|<\infty$. Then
\begin{equation}\label{eq:lp_dft}
	\max_{1\le j\le q}\|\mathcal X_n(\omega_j)-A(\omega_j)\mathcal E_n(\omega_j)\|=o_P(\log^{-1/2}n)
\end{equation}
and
\begin{equation}\label{eq:lp_pdg}
	\max_{1\le j\le q}\|\mathcal X_n(\omega_j)\otimes\mathcal X_n(\omega_j)
	-A(\omega_j)\mathcal E_n(\omega_j)\otimes A(\omega_j)\mathcal E_n(\omega_j)\|
	=o_P(1)
\end{equation}
as $n\to\infty$, where $A(\omega)$ is given by \eqref{eq:A(omega))} for $\omega\in[\pi,\pi]$.
\end{lem}
We note that we require a weaker summability condition than in \citet{walker:1965}, where it is assumed that $\sum_{k\ne 0}|k|^{1/2}\|a_k\|<\infty$.

\autoref{lem:lp} implies that
\begin{equation}\label{e:diffdft}
	\max_{1\le j\le q}\|\mathcal X_n(\omega_j)\|^2-\max_{1\le j\le q}\|A(\omega_j)\mathcal E_n(\omega_j)\|^2
	=o_P(1)
	\quad\text{as}\quad n\to\infty.
\end{equation}
With additional assumptions on $A(\omega)$ it is possible to establish the asymptotic distribution of  $\max_{1\le j\le q}\|A^{-1}(\omega_j)\mathcal X_n(\omega_j)\|^2$ from $\max_{1\le j\le q}\|\mathcal E_n(\omega_j)\|^2$.

\begin{lem}\label{lem:lpfilter}
Suppose that $\{X_t\}_{t\in\mathbb Z}$ is given by \eqref{eq:linproc}, $\sum_{k\ne0}\log(|k|)\|a_k\|<\infty$, $A^{-1}(\omega)$ exists for each $\omega\in[-\pi,\pi]$ and $\sup_{\omega\in[0,\pi]}\|A^{-1}(\omega)\|<\infty$, where $A(\omega)$ is given by \eqref{eq:A(omega))}. Then
\[
	\max_{1\le j\le q}\|A_n^{-1}(\omega_j)\mathcal X_n(\omega_j)\|^2-\max_{1\le j\le q}\|\mathcal E_n(\omega_j)\|^2
	=o_P(1)
	\quad\text{as}\quad n\to\infty.
\]
\end{lem}
The following example illustrates that the assumptions of \autoref{lem:lpfilter} are satisfied by an FAR(1) model.
\begin{ex}
Consider an FAR(1) model given by
\[
	X_t
	=\rho(X_{t-1})+\varepsilon_t
	=\sum_{j=0}^\infty\rho^j(\varepsilon_{t-j})
\]
for $t\in\mathbb Z$ with $\rho\in \mathcal{L}(H_0)$ such that $\|\rho^{n_0}\|<1$ with some $n_0\ge1$ (see Chapter~3 of \citet{bosq:2000} for more details). Since $A(\omega)$ is a Neumann series for each $\omega\in[-\pi,\pi]$, we have that $A(\omega)=(I-e^{-\ii\omega}\rho)^{-1}$ and hence $A^{-1}(\omega)=I-e^{-\ii\omega}\rho$ exists for each $\omega\in[-\pi,\pi]$, and $\sup_{\omega\in[0,\pi]}\|A^{-1}(\omega)\|<\infty$.
\end{ex}

\autoref{lem:lpfilter} allows us to obtain the following theorem.
\begin{Theorem}\label{cor:lpinfty}
Suppose that $\{X_t\}_{t\in\mathbb Z}$ is given by \eqref{eq:linproc}, the assumptions of \autoref{lem:lpfilter} are satisfied and $\{\varepsilon_t\}_{t\in\mathbb Z}$ satisfy the assumptions of \autoref{fancythm}. Then
\[
	\lambda_1^{-1}\Bigl(\max_{1\le j\le q}\|A^{-1}(\omega_j)\mathcal X_n(\omega_j)\|^2-b_q\Bigr)\convd\mathcal G
	\quad\text{as}\quad n\to\infty.
\]
 The eigenvalues $\lambda_1$ and those in the definition of $b_n$ are the eigenvalues of the covariance operator $\operatorname{Var}(\varepsilon_0)$.
\end{Theorem}

If we restrict our attention to the multivariate case, i.e.\ $H_0=\mathbb R^d$, then we can standardize the covariance structure of $\{\varepsilon_t\}_{t\in\mathbb Z}$. We have the following result in the finite dimensional setting. We note that in the following theorem we do not require distinct eigenvalues of $\operatorname{Var}(\varepsilon_0)$ as long as they all are positive.
\begin{Theorem}\label{thm:mv}
Suppose that $H_0=\mathbb R^d$, $\{X_t\}_{t\in\mathbb Z}$ is given by \eqref{eq:linproc}, $\sum_{k\ne0}\log(|k|)\|a_k\|<\infty$, $A^{-1}(\omega)$ exists for each $\omega\in[-\pi,\pi]$ and $\sup_{\omega\in[0,\pi]}\|A^{-1}(\omega)\|<\infty$, where $A(\omega)$ is given by \eqref{eq:A(omega))} for $\omega\in[-\pi,\pi]$. Suppose that the covariance matrix $\Sigma:=\operatorname{Var}(\varepsilon_0)$ is positive definite. Then
\[
	\max_{1\le j\le q}\|B^{-1}(\omega_j)\mathcal X_n(\omega_j)\|^2-c_n
	\xrightarrow{d}\mathcal G
	\quad\text{as}\quad n\to\infty,
\]
where $B(\omega)=A(\omega)\Sigma^{1/2}$, $c_n$ is given by \eqref{eq:stand:c_n}.
\end{Theorem}
We conclude by remarking that the spectral density matrix can be expressed as
\begin{equation}\label{eq:sdensity}
	F(\omega)
	=A(\omega)\Sigma A^*(\omega)
	=B(\omega)B^*(\omega).
\end{equation}
for $\omega\in[-\pi,\pi]$. Hence, we have that
\[
	\|B^{-1}(\omega)\mathcal X_n(\omega)\|^2
	=\operatorname{Tr}[F^{-1}(\omega)[\mathcal X_n(\omega)\otimes\mathcal X_n(\omega)]]
\]
for $\omega\in[-\pi,\pi]$.

\section{Detecting periodic signals}\label{test}
In this section we discuss the application of our results to testing for hidden periodicities in functional time series. Our basic framework hence is the following:  assume that the sequence $\{Y_t\}_{t\in\mathbb Z}$ is given by
\begin{equation}\label{eq:model}
	Y_t
	=\mu+s(t)+X_t 
\end{equation}
for $t\in\mathbb Z$, where $\mu\in H_0$, $s:\mathbb Z\to H_0$ is a deterministic periodic function such that $s(t)=s(t+d)$ for all $t\in\mathbb Z$ with some $d\geq 2$ and $\sum_{t=1}^ds(t)=0$. We complement the recent results of \citet{hoermann:2018}, where such tests were developed when the length of the period $d$ is assumed to be known. In the following we do not assume that $d$ is known.
We investigate the subsequent testing problem:
\begin{equation}\label{test0}
	\mathcal{H}_0\colon \text{\eqref{eq:model} holds with $\|s(t)\|\equiv 0$}
	\quad\text{versus}\quad
	\mathcal{H}_1\colon \text{\eqref{eq:model} holds with $\|s(t)\|\not\equiv 0$.}
\end{equation}
The noise process $X_t$ can follow any of the different settings discussed in the present paper (multivariate, multivariate with increasing dimension, iid data, linear processes). Of course, every setting requires different---though conceptually similar---test statistics. To keep the paper streamlined we focus here on the infinite dimensional setting. In particular we are going to assume that $X_t$ is an  FAR(1) process $X_t=\rho(X_{t-1})+\varepsilon_t$. For this setup we will work out the details. With $\rho=0$ this includes the iid case, where we can actually relax \autoref{ass:rho} below, since we do not have to estimate $\rho$ then. The proofs of this section are given in Section~\ref{s:FAR1}.

Suppose for the moment that $\Sigma=\mathrm{Var}(\varepsilon_t)$ and $\rho$ are known.  Let $\lambda_j$ be the eigenvalues of $\Sigma$. Then, under $\mathcal{H}_0$ and suitable assumptions on the innovations $\varepsilon_t$, we get by \autoref{cor:lpinfty} that the test statistic 
\[
\lambda_1^{-1}\max_{1\le j\le q}\|(I-e^{-\ii\omega_j}\rho)\mathcal Y_n(\omega_j)\|^2-\log(q)+\sum_{j=2}^\infty\log(1-\lambda_j/\lambda_1)
\]
converges to the standard Gumble distribution. Here $\mathcal Y_n(\omega_j)$ denotes the discrete Fourier transform of $Y_1,\ldots,Y_n$ (note that under $\mathcal{H}_0$ we have $\mathcal Y_n(\omega_j)=\mathcal X_n(\omega_j)$ for all $1\leq j\leq q$.) In practice, we need to replace $\rho$ and $\lambda_j$ by estimators to get a valid test statistic. We will impose the following assumption. 
\begin{assumption}\label{ass:rho}
Suppose that $\widehat\rho$\, is an estimator of $\rho$, with
$
\|\widehat\rho-\rho\|=o_P(a_n^{-1})$, where $\log n\leq a_n\leq \sqrt{n}$. Assume $\|\rho\|<1$. Assume moreover, that the innovations $\varepsilon_t$ satisfy the assumptions of \autoref{fancythm}. Finally we suppose that $\mu=0$.
\end{assumption}

\autoref{ass:rho} contains the basic assumptions on the innovations which we require in the iid case to apply our theorems. In addition we need a consistent estimator for $\rho$, which is, for example, established in \citet{bosq:2000} or \citet{hk:2015}. Rates of convergence can be found  in \citet{guillas:2001}. The requirement $\|\rho\|<1$ assures that the corresponding FAR(1) process is stationary. Assuming $\mu=0$ is a simplification. Otherwise we center the data by the sample mean. A constant shift does not alter $\mathcal Y_n(\omega_j)$ for $j=1,\ldots, q$.

\begin{Theorem}\label{th:appl}
 Define $\hat\lambda_j$ to be the eigenvalues of $\frac{1}{n-1}\sum_{k=2}^n\hat\varepsilon_k\otimes\hat\varepsilon_k$, where 
$$
\hat\varepsilon_k=X_k-\widehat\rho\,(X_{k-1}), \quad k=2,\ldots,n.
$$
Under $\mathcal{H}_0$ and \autoref{ass:rho}, we have that
\[
T_n\coloneqq\hat\lambda_1^{-1}\max_{1\le j\le q}\|(I-e^{-\ii\omega_j}\widehat\rho\,)\mathcal Y_n(\omega_j)\|^2-\log(q)+\sum_{j=2}^{a_n}\log(1-\hat\lambda_j/\hat\lambda_1)\convd\mathcal G
\quad\text{as}\quad n\to\infty.
\]
\end{Theorem}
\begin{rmrk}
In \autoref{th:appl} the truncation parameter $a_n$ in the centering constant can be replaced by any $b_n\leq a_n$ with $b_n\to\infty$.
\end{rmrk}
Our next result establishes consistency of our test statistic when $\mathcal{H}_0$ is violated. We assume that there exists $d\geq 2$ such that $s(t)=s(t+d)$ and $\|s(t)\|\not\equiv 0$.  In the formulation of the theorem below, we allow $d$ and $s(t)$ to be dependent on $n$.

\begin{Theorem}\label{th:applconsist}
Consider the assumptions of \autoref{th:appl}, but assume now that $\mathcal{H}_0$ doesn't hold. Suppose that $\max_{1\leq t\leq d}\|s(t)\|=O(1)$ and
\begin{equation}\label{e:applcons}
\frac{\sqrt{n}}{d^2}\to\infty\quad \text{and}\quad\psi_n\coloneqq\biggl\|\sum_{t=1}^d s(t)e^{-\ii \frac{2\pi}{d} t}\biggr\|\frac{1}{d^2}\sqrt{\frac{n}{\log n}}\to\infty. 
\end{equation}
Suppose moreover, that $\widehat\rho\convP \rho^\prime$, with $\|\rho^\prime\|<1$ and $\hat\lambda_j\convP \lambda_j^\prime$, with $\sum_{j\geq 1}\lambda^\prime_j<\infty$. Then we have $T_n\convP\infty$ as $n\to\infty$.
\end{Theorem}
Condition \eqref{e:applcons} is a technical condition which is fairly mild and which assures that the periodic signal is strong enough to be picked up by the Fourier transform. The assumptions on $\widehat\rho\,$ and $\hat\lambda_j$ are needed because the violation of $\mathcal{H}_0$ implies that our process $\{Y_t\}_{t\in\mathbb{Z}}$ is not stationary. Therefore the estimator for $\rho$---neglecting the underlying periodic signal---is in general not consistent. When the length of the period is known, then the estimator can be adapted to remain consistent under the alternative. Here we do not assume that $d$ is known and hence we use the same estimator for $\rho$ as in the stationary case. To work out the asymptotics of the estimator under the alternative is beyond the scope of this paper and hence is phrased as an assumption.

\section{Empirical study}\label{s:empirical}
In this section we compare the asymptotic theory developed in this paper to the finite sample behaviour of the statistic $T_n$ from Section~\ref{test}.
To this end we organize a simulation study which we describe now in detail. The first step is to generate suitable data. 

\subsection{Generating functional time series}\label{s:datagen}
The target in a simulation is to generate synthetic data, so that we have control over the data generating process (DGP). Often, however, we find the available simulation settings for functional data rather unrealistic. We want to explain here a setting which allows to generate synthetic and at the same time realistic data. To this end we use as our basic building block a real data set which we are well familiar with and which we have used as a toy data set in different papers, namely \texttt{PM10} curves in Graz, Austria. \texttt{PM10} is measured in $\mu g/m^3$ and describes the amount of particles with a diameter of less the 10$\mu m$ in 1 cubic-meter of air. Specifically, our data set consists of 182 observation days in the winter season 2010/2011 (October--March). The data are recorded in 30 minutes intervals, resulting in 48 observations per day. We have removed the week around New Year's Eve because of high outlying observations due to fireworks, leaving 175 days.  In the data preprocessing we have also removed a potential weekday effect, by centering the data with corresponding weekday averages.  To account for heavy tails, we have done a square-root transformation, i.e.\ we look at $\sqrt{\texttt{PM10}}$. The preprocessed data are than transformed to functional data by a basis function approach, see \citet{ramsay09}. We use the \texttt{R}-package \texttt{fda} and the command \texttt{Data2fd} with 21 Fourier basis functions. To  the resulting functional time series $Z_1,\ldots, Z_{175}$ we fit an FAR(1) model $Z_t=\psi(Z_{t-1})+e_t$. The estimator $\widehat\psi$ is a PCA based estimator defined as in \citet{bosq:2000}, p.~218. 
We set the tuning parameter $k_n=8$. This parameter determines the number of principal components to use for the estimator. In our example 8 principal components are needed to explain more than $99\%$ of the variability in the data. 
In general, a linear operator $\psi$ on the function space $L^2$ can be represented in the form $\psi=\sum_{i,j\geq 1}\psi_{i,j}v_i\otimes v_j$ where $\{v_i\colon i\geq 1\}$ are  the Fourier basis functions. Hence $\psi$ is equivalent to an infinite dimensional correspondance matrix $\Psi=((\psi_{ij}))$. In our case, since we use 21 Fourier basis functions to expand the data, $\widehat\psi$ corresponds to a $21\times 21$ matrix $\widehat\Psi$. In \autoref{fig:psi} we show the $9\times 9$ sub-matrix representing the upper left corner of $\widehat\Psi$. This $\widehat\Psi$ is close to an upper triangular matrix. It is very different from common settings where mainly diagonal or symmetric matrices are used.

\begin{figure}[h]
\captionsetup{width=0.8\textwidth}
\centering
\includegraphics[width=10cm]{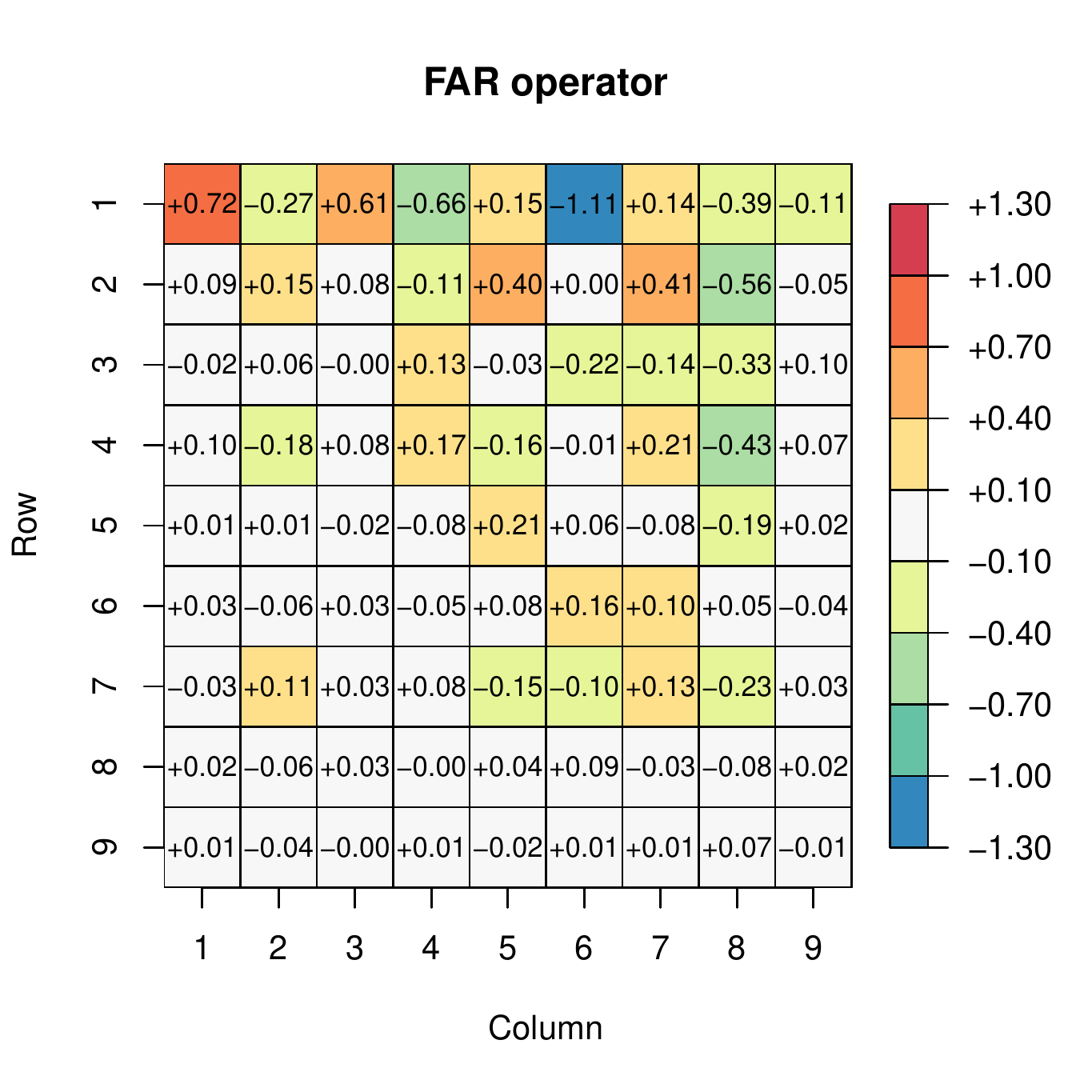}
\caption{The coefficient matrix (upper $9\times 9$ elements)  of the FAR(1) operator estimated for the \texttt{PM10} sample and used for our DGP.}
\label{fig:psi}
\end{figure}

Now we start with the actual generation of our synthetic data. To this end we compute the residuals $\hat e_t=Z_t-\widehat\psi\,(Z_{t-1})$, $2\leq t\leq 175$ and generate a functional time series $X_t=\rho(X_{t-1})+\varepsilon_t$, using $\rho=\widehat \psi$, and $\varepsilon_0,\ldots,\varepsilon_n$ being an iid bootstrap sample of size $n$ from $\hat e_2,\ldots, \hat e_{175}$. We use $X_0=\varepsilon_0$. Our construction assures that we get a functional time series which is stationary and behaves similarly as the original \texttt{PM10} data.

\subsection{Simulation setting}
The core algorithm for our simulations can be described as follows:\medskip

\noindent
\underline{{\bf Simulation algorithm:}}

\begin{enumerate}
\item Generate $n$ data from the FAR(1) process $X_t=\rho(X_{t-1})+\varepsilon_t$.
\item Generate  a $d$-periodic signal $s(t)$ and define $Y_t=s(t)+X_t$. 
\item Estimate the auto-regression operator $\rho$.
\item Calculate the residuals $\hat\varepsilon_t=X_t-\widehat \rho\,(X_{t-1})$.
\item Using $\hat\varepsilon_t$ compute estimates $\hat\lambda_j$ for the eigenvalues of $\Sigma=\mathrm{Var}(\varepsilon_0)$.
\item Compute $T_n$ and then $\delta:=I_{\{T_n>q_{1-\alpha}\}}$, where $q_{1-\alpha}=\mathcal{G}^{-1}(1-\alpha)$ and $I_A$ is the indicator function on $A$.
\item Repeat Steps 1--6 $2000$ times independently to obtain $\delta_1,\ldots, \delta_{2000}$ and calculate the empirical rejection rate $\hat r:=\mathrm{av}(\delta_i\colon 1\leq i\leq 2000)$.
\end{enumerate}

Step~1 was outlined in Section~\ref{s:datagen}. For the sample sizes we use $n=100,250,500$. The periodic signal in Step~2 we define as $s(t,u)=s(t)=a\cos(2\pi t/d)$, where $d-2$ is a Poisson-distributed random variable $P_\lambda$ with $\lambda=5$ and $\lambda=15$. (Note that we guarantee $d\geq 2$.)
For $a$ we investiage the values $a=0,1,2$. Clearly, $a=0$ corresponds to $\mathcal{H}_0$. In Step~3 we estimate $\rho$ using the estimator outlined in Section~\ref{s:datagen}, again with $k_n$ such that we explain more than $99\%$ of the variance in our sample. In Step~6 we need to choose $a_n$. We use $a_n=\text{argmin}_{j\geq 1}\{-\log(1-\hat\lambda_j/\hat\lambda_1)\leq 0.01\}$. The significance levels for our tests are $\alpha\in \{0.1,0.05,0.01\}$.

For all combinations of $n$, $\lambda$ and $a$ we run the experiment 2000 times and report $\hat r=\hat r(n,\lambda,a,\alpha)$  in \autoref{tab:1}. We can see that the respective size is captured fairly accurately even at the relatively small sample size $n=100$. Not surprisingly, the test is more powerful for shorter periods and larger sample sizes. Concerning the power we notice that under our setting with  $a=1$  the signal-to-noise ratio is
$$
\frac{1}{d}\sum_{t=1}^d\|s(t)\|^2\Big/E\|X_t\|^2\approx \frac{1}{5.5}.
$$
Here we have approximated $E\|X_t\|^2$ by $\frac{1}{n}\sum_{t=1}^n\|X_t\|^2$ with $n=10^4$. 

\begin{table}[ht]
\centering
\begin{tabular}{l|r|rrr|rrr|rrr}
\multicolumn{2}{c|}{$\hat r(n,\lambda,a,\alpha)$} &  \multicolumn{3}{c}{$a=0$ ($\equiv \mathcal{H}_0$)} & \multicolumn{3}{|c|}{$a=1$} &\multicolumn{3}{|c}{$a=2$}\\
  \hline\hline
  & $\alpha$ & 0.1 & 0.05 & 0.01 & 0.1 & 0.05 & 0.01 & 0.1 & 0.05 & 0.01 \\ 
  \hline\hline
$\lambda=5$ &$n=100$ & 0.066 & 0.029 & 0.004 & 0.861 & 0.799 & 0.670 & 1.000 & 0.999 & 0.993 \\  
 & $n=200$ &0.082 & 0.038 & 0.006 & 0.989 & 0.983 & 0.970 & 1.000 & 1.000 & 1.000 \\ 
  &$n=500$ & 0.093 & 0.054 & 0.011 & 1.000 & 1.000 & 0.999 & 1.000 & 1.000 & 1.000 \\ 
   \hline
 $\lambda=15$ &  $n=100$ & 0.082 & 0.041 & 0.005 & 0.249 & 0.165 & 0.071 & 0.818 & 0.758 & 0.606 \\ 
  &$n=200$ & 0.071 & 0.035 & 0.006 & 0.569 & 0.471 & 0.293 & 0.985 & 0.973 & 0.922 \\  
  &$n=500$ &0.096 & 0.045 & 0.007 & 0.990 & 0.978 & 0.942 & 1.000 & 1.000 & 1.000 \\ 
\end{tabular}
\caption{Empirical rejection rates in our simulation study.}
\label{tab:1}
\end{table}

\subsection{Application to real data}
We now apply the test directly to the \texttt{PM10} data set. In \citet{hoermann:2018} the same data were tested for a fixed period $d=7$ in order to reveal a potential weekday effect. It was found there, that such a weekday effect is significant. The reason being that on weekends the shape of the \texttt{PM10} curves (again we use $\sqrt{\texttt{PM10}}$ curves) changes towards a lower level during day time and higher levels during the night time. Since the test we propose here is not requiring knowledge of the period $d$, it is of course expected to have smaller power.

We consider two settings: in the first we use the data $Z_1,\ldots, Z_{175}$ as described in Section~\ref{s:datagen}, i.e.\ the detrended data, centered by the weekday averages. In addition we consider $\tilde Z_1,\ldots, \tilde Z_{175}$, where  the detrending step is skipped. This data corresponds to the actual $\sqrt{\texttt{PM10}}$ curves.

Instead of plainly computing the test statistic $T_n$ we rather show in \autoref{fig:test2} plots of
$$
T_n(j):=\|(I-e^{-\ii\omega_j}\widehat\rho\,)\mathcal Y_n(\omega_j)\|^2-\log(q)+\sum_{j=2}^{a_n}\log(1-\hat\lambda_j/\hat\lambda_1)\quad j=1,\ldots, q=87.
$$
The horizontal lines represent critical values at levels $\alpha=0.1$, $0.05$ and $0.01$. For the detrended data (left figure) we cannot find a significant violation of $\mathcal{H}_0$. Also for the non-detrended data (right figure) a weekly periodicity, corresponding to frequency $\omega_{25}=\frac{2\pi}{7}$ (marked by the dashed vertical line) does not stand out significantly. So here we are confronted with the loss in power we mentioned before. However, to our surprise, we did notice a significant periodicity at frequency $\omega_1$. A closer look into the data shows that it can be explained by a seasonal behavior of \texttt{PM10}, which we did not notice earlier. Taking a moving averages sliding over the data, we observed a slightly increasing trend of the base \texttt{PM10} level towards the high winter, followed again by a decreasing trend towards spring. We remark that it is quite difficult to notice such features by visual inspection, since plotting and visually analysing 175 functional data in a sequence is not quite obvious.

In practice it is advisable to test for a fixed frequency, if we have a particular conjecture about the length of the period. The example shows that it is well worth to complement this approach with our new test, as it may reveal periodicities which are not a priori expected.

\begin{figure}[h]
\captionsetup{width=0.8\textwidth}
\centering
\includegraphics[width=7cm]{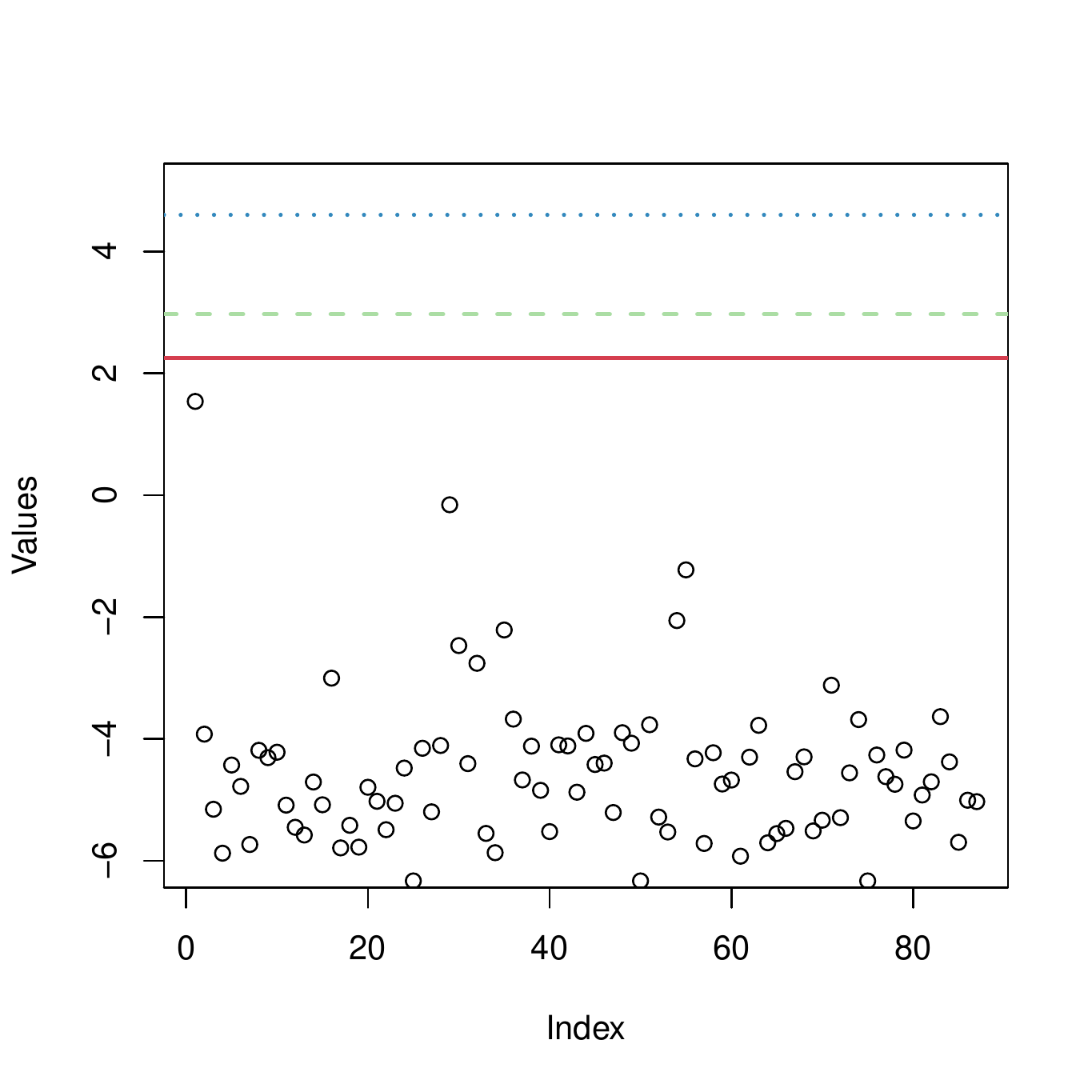}
\includegraphics[width=7cm]{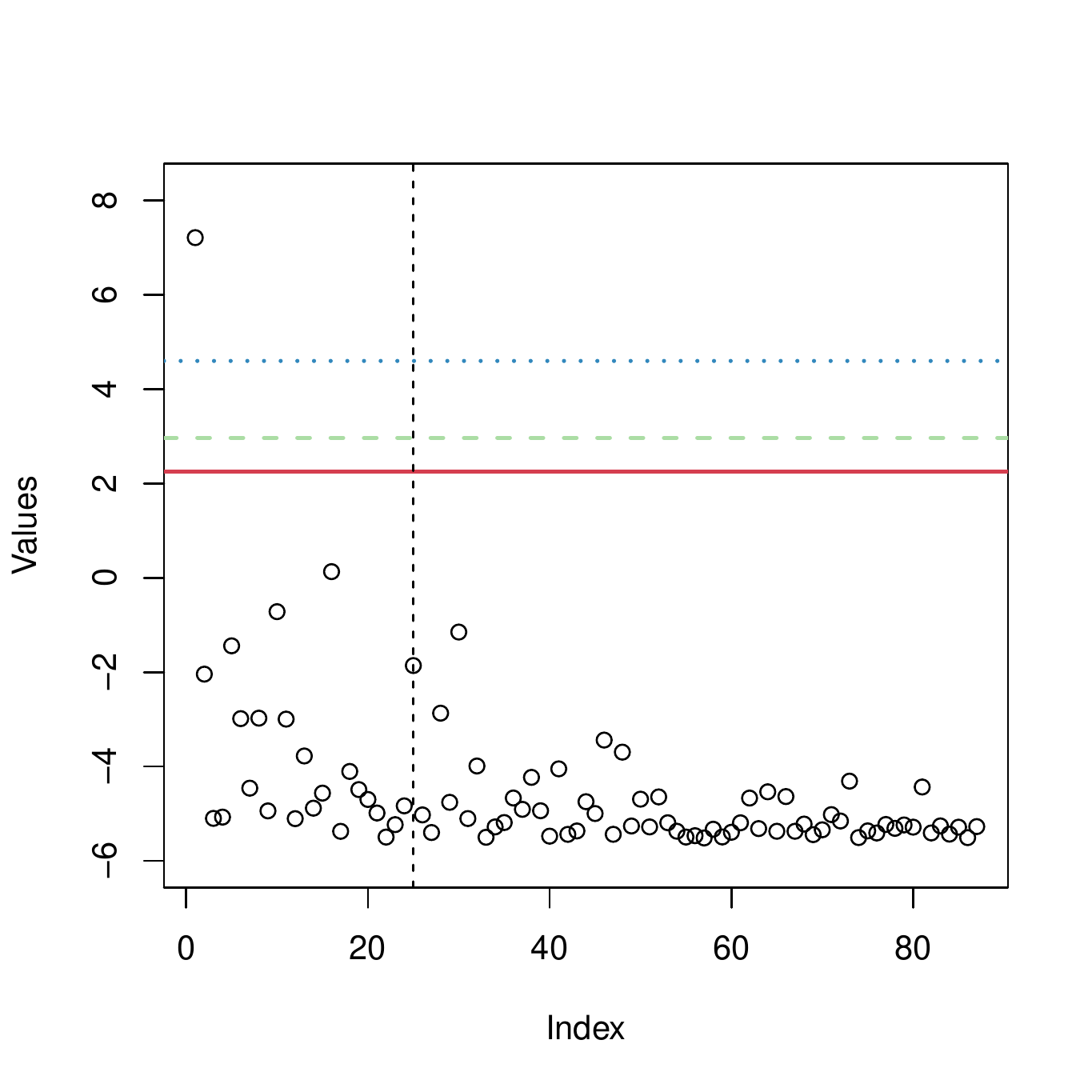}
\caption{The statistics $T_n(j)$ plotted for index $j=1,\ldots, q$. The left figure is based on $(Z_i)$ (detrended data) the right figure on $(\tilde Z_i)$ (actually \texttt{PM10} curves).}
\label{fig:test2}
\end{figure}

\section{Conclusion}\label{s:conclusion}

We have investigated the limiting distribution of the maximum norm of the periodogram operator of a Hilbert space valued random sequence. This a very useful statistic when we are interested in revealing a hidden periodic signal in functional time series. For the proof of our main results we proceed stepwise from the multivariate, to the high-dimensional (i.e.\ the dimension is diverging with sample size) and then to the infinite dimensional case. The method of proof we use is based on recent advances in the normal approximation of high dimensional data in \citet{chernozhukov:2017}. Our approach can be used to recover a classical result of \citet{davis:mikosch:1999} for univariate data. In fact, the proof for the univariate results in \citet{davis:mikosch:1999} with our approach would be much shorter. For passing to the infinite dimensional case we had to slightly adapt the result of  \citet{chernozhukov:2017} for our needs, and make the constant in the normal approximation bounds explicit. The application also demands to extend our theory beyond independent data. We have presented an extension to linear processes under quite sharp conditions.

Finally, we conducted an empirical study to investigate how this theory works with simulated as well as real data. We investigate the \texttt{PM10} data set from Graz, Austria  which we are well familiar with and which we have used as an example in different publications (see \citet{stadlober:2008} and \citet{hoermann:2018}). This is an air quality data set that contains the amount of particulate matter of up to \SI{10}{\micro\metre} in diameter measured in \SI{}{\micro\gram/\metre^3}.
The \texttt{PM10} data set is also the main building block of our simulated data. We use it to generate synthetic and at the same time realistic data via some resampling scheme. Our simulation study shows that our approach has good finite sample performance. We also compare our test with the test of \citet{hoermann:2018} using the \texttt{PM10} data set. Since here we do not require the knowledge of the period, it is expected that we have smaller power. Our test does not detect the same (weekly) periodic component as the test by \citet{hoermann:2018}, but the new approach reveals another seasonal effect which we did not notice previously. 

\section{Proofs and auxiliary lemmas}\label{proofsart3}
For the proofs we introduce the following notation and conventions. We use again $\|\cdot\|$ as norm on $H$ but also for the Euclidian norm in $\mathbb{R}^d$. The specific meaning should be clear from the context. We use $N_d(\mu,\Sigma)$ for the $d$-variate normal law with mean $\mu$ and covariance $\Sigma$. The unit-sphere in $\mathbb{R}^d$ is denoted $\mathbb S^{d-1}=\{x\in\mathbb R^d:\|x\|=1\}$. We define $\|u\|_0$ the number of non-zero components of the vector $u\in\mathbb R^d$. We will use $I_A$ for indicator function of a set $A$ and $I_d$ the identity matrix in $\mathbb{R}^d$.

The main tool of our proofs is a powerful result of~\citet{chernozhukov:2017}. Suppose that $V_1,\dots,V_n$ are independent random vectors in $\mathbb R^p$ with zero means and finite second moments. Let $W_1,\dots,W_n$ be independent Gaussian random vectors in $\mathbb R^p$ such that $W_i\sim N_p(0,E[V_iV_i'])$ for $1\le i\le n$. Set  $S^{V}_n=n^{-1/2}\sum_{i=1}^n V_i$ and $S^{W}_n=n^{-1/2}\sum_{i=1}^n W_i$ for $n\ge1$. \citet{chernozhukov:2017} establish a bound for 
\begin{equation}\label{e:cbound}
	\rho_n(\mathcal{A}^{\mathrm{sp}}(s)) = \sup_{A\in \mathcal{A}^{\mathrm{sp}}(s)}|P(S^{V}_n \in A)-P(S^{W}_n \in A)|,
\end{equation}
where $\mathcal{A}^{\mathrm{sp}}(s)$ is the class of $s$-sparsely convex subsets of $\mathbb{R}^d$. A set $A$ is an element of  $\mathcal{A}^{\mathrm{sp}}(s)$ if $A$ is an intersection of finitely many convex sets $A_k$ and if the indicator function of each $A_k$, $x\mapsto I_{A_k}(x)$, depends only on $s$ components of its argument $x=(x_1,\ldots, x_d)$. We state some conditions which will be needed:
\begin{itemize}
\item[(i)] $n^{-1}\sum_{t=1}^n E|u'V_t|^2\ge b$  for all  $u\in \mathbb{S}^{p-1}$ and $\|u\|_0\leq s$;
\item[(ii)]$n^{-1}\sum_{t=1}^n E|V_{t,j}|^{2+k}\leq B_n^k$ for all $j=1,\ldots,p$ and $k=1,2$;
\item[(iii)]$E\exp(|V_{t,j}|/B_n)\leq 2$ for all $t=1,\ldots,n$ and $j=1,\ldots,p$;
\item[(iv)]$E\max_{1\le j\le p}(|V_{t,j}|/B_n)^q\le 2$ for all $t=1,\ldots,n$,
\end{itemize}
where $b,q>0$ are some constants and $B_n\geq 1$ is a sequence of constants, possibly growing to infinity as $n\to\infty$. 
\begin{Proposition}(\citet[Proposition~3.2]{chernozhukov:2017}))\label{p:cher}
Under conditions (i), (ii) and (iii), it holds that
\begin{equation}\label{cherrho}
 \rho_n(\mathcal{A}^{\mathrm{sp}}(s)) \leq C\cdot  
\frac{B_n^{1/3}\log^{7/6}(pn)}{n^{1/6}}
\end{equation}
for $n\ge4$. The constant $C$ in \eqref{cherrho} depends only on $b$, $s$ and $q$. 
\end{Proposition}

\subsection{Proofs of main results}

\begin{proof}[Proof of \autoref{mainthm0}]
We denote
\[
	\tilde{X}_t
	=X_t I_{\{\|X_t\|\le n^{1/r}\}}-\operatorname E[X_1 I_{\{\|X_1\|\le n^{1/r}\}}],
\]
and
\[
	\tilde X_t^d
	=\sum_{k=1}^d\langle\tilde X_t,v_k\rangle v_k
	\quad\text{and}\quad
	\tilde{\mathcal X}_n^d(\omega)
	=n^{-1/2}\sum_{t=1}^n\tilde X_t^de^{-it\omega}
\]
for $n\ge1$, $t\ge1$, $d\ge1$ and $\omega\in[-\pi,\pi]$. 
In view of \autoref{lemma:trunc} in Section~\ref{sec:aux}, it suffices to show that $\lambda_1^{-1}(\max_{1\leq j\leq q}  \|\tilde{\mathcal{X}}_n^d(\omega_j)\|^2-b_n^d)\convd\mathcal G$ as $n\to\infty$. To this end let us define $\mathbb R^{2dq}$-valued random vectors
\[
	\D
	=n^{-1/2}\sum_{t=1}^n V_t
\]
for $n\ge1$, where
$
	V_t
	=(\langle\tilde X_t,v_1\rangle f_t',\ldots,\langle\tilde  X_t,v_d\rangle f_t')'
$
with
\begin{equation}\label{eq:Wt2q}
	f_t
	=(\cos(t\omega_1),\sin(t\omega_1),\ldots,\cos(t\omega_q),\sin(t\omega_q))'\in\mathbb{R}^{2q}.
\end{equation}
Note,  that for the sake of a lighter notation  we suppress in some variables the dependence on $d$ and $q$ .

We let $V_{t,m}$ and $\Dsub{m}$ be the $m$-th element of the vectors $V_t$ and $\D$, respectively. For some ordered index set $J$ we let $V_{t,J}=(V_{t,j}\colon j\in J)'$. Analogously we define $\Dsub{J}$. Then 
$$V_{t,2q\times (\ell-1)+2k-1}=\langle \tilde X_t,v_\ell\rangle \cos(t\omega_k)\quad\text{and}\quad 
V_{t,2q\times (\ell-1)+2k}=\langle \tilde X_t,v_\ell\rangle \sin(t\omega_k),$$
for $1\leq \ell\leq d$ and $1\leq k\leq q.$
Thus, with the sets $J_k=J+2(k-1)$, where
\begin{equation}\label{eq:J}
	J=\{1,2,2q+1,2q+2,\dots,2(d-1)q+1,2(d-1)q+2\}	
\end{equation}
we obtain vectors $V_{t,J_k}\in\mathbb{R}^{2d}$, $k=1,\ldots, q$, where
$$
V_{t,J_k}=(\langle\tilde X_t,v_1\rangle \cos(t\omega_k),\langle\tilde X_t,v_1\rangle \sin(t\omega_k),\ldots, 
\langle\tilde X_t,v_d\rangle \cos(t\omega_k),\langle\tilde X_t,v_d\rangle \sin(t\omega_k))'.
$$ It holds that
\begin{align*}
 	P(\max_{1\leq k\leq q}\|\tilde{\mathcal{X}}_n^d(\omega_k)\|^2\leq x)
 	&=P( \|\tilde{\mathcal X}_n^d(\omega_k)\|^2\leq x\ \text{for all}\ k=1,\ldots,q )\\
	&=P(\|\Dsub{J_k}\|^2\leq x\ \text{for all}\ k=1,\ldots,q)\\
	&=P(\D\in \cap_{k=1}^q A_k),
\end{align*}
where 
\[
	A_k=\{y\in\mathbb R^{2dq}:\|(y_j)_{j\in J_k}\|^2\le x\}.
\]
It is important to note that $\cap_{k=1}^q A_k$ is a $2d$-sparsely convex set. 

Our target is then to apply Proposition~\ref{p:cher}. To this end we show that conditions (i), (ii) and (iii) hold. Suppose that $u\in\mathbb S^{2dq-1}$ and $u=(u_1' , \ldots ,u_d')'$ with $u_\ell\in\mathbb R^{2q}$. We obtain
\begin{align*}
	n^{-1}\sum_{t=1}^n\operatorname E|\langle V_t,u\rangle|^2
	&\ge n^{-1}\sum_{t=1}^n\operatorname E\Bigl|\sum_{k=1}^p\langle X_t,v_k\rangle f_t' u_k\Bigr|^2\\
	&+2n^{-1}\sum_{t=1}^n\sum_{k,l=1}^d\operatorname E[\langle X_t,v_k\rangle\langle \tilde X_t-X_t,v_l\rangle]f_t'u_k f_t' u_l\\
	&=T_1+T_2.
\end{align*}
We have that $T_1\ge\lambda_d/2$ (see \autoref{lemM1} in Section~\ref{sec:aux}). Since $n^{-1}\sum_{t=1}^n f_t f_t'=\frac{1}{2}I_{2q}$ and $\sum_{k,l=1}^d\langle u_l,u_k\rangle\le d^2$, we obtain
\[
	|T_2|
	\le\sum_{k,l=1}^d\operatorname E|\langle X_1,v_k\rangle\langle\tilde X_1-X_1,v_l\rangle||\langle u_l,u_k\rangle|
	\le d^2(\operatorname E\|X_1\|^2)^{1/2}(\operatorname E\|\tilde X_1-X_1\|^2)^{1/2}.
\]
Clearly $\operatorname E\|\tilde X_1-X_1\|^2\to0$, and hence $n^{-1}\sum_{t=1}^n\operatorname E|\langle V_t,u\rangle|^2\ge\lambda_d/2+o(1)$ as $n\to\infty$ and thus condition (i) is satisfied. 

To verify (ii) we first notice that for any $\ell$ and $m$
$$
|\langle \tilde X_t,v_\ell\rangle| \max\{|\cos(t\omega_m)|,|\sin(t\omega_m)|\}
\leq 2\|X_t\|I_{\{\|X_1\|\le n^{1/r}\}}.
$$ 
Hence, if $r<2+k$, then
\[
	n^{-1}\sum_{t=1}^n\operatorname E|V_{t,j}|^{2+k}
	\le2^{2+k}\operatorname E[\|X_1\|^{2+k}I_{\{\|X_1\|\le n^{1/r}\}}]
	=O(n^{(2+k)/r-1})
\]
as $n\to\infty$. For $r>2+k$, we have $n^{-1}\sum_{t=1}^n\operatorname E|V_{t,j}|^{2+k}=O(1)$. Thus, (ii) is satisfied if we set $B_n=cn^{1/r}$ for $n\ge1$ with some $c>0$. Finally, (iii) follows from
\[
	\operatorname E\exp(|V_{t,j}|/B_n)
	\le\exp(2n^{1/r}/B_n)\leq 2,
\]
if $c\ge 2/\log2$. 

Hence, (i), (ii) and (iii) hold, which in turn implies that \eqref{cherrho} holds with $B_n=cn^{1/r}$, where $c\ge 2/\log2$ and $r>2$. The bound in \eqref{cherrho} tends to 0 with $n\to\infty$. 

Suppose now that $\tilde Y_1,\tilde Y_2,\ldots$ are iid Gaussian random elements with values in $H_0$ such that $\operatorname E\tilde Y_1=0$ and $\operatorname{Var}(\tilde Y_1)=\operatorname{Var}(\tilde X_1)$. Then the $2dq$-dimensional random vectors $V_t$ have the same covariance matrices as the Gaussian random vectors $W_t$, where
\[
	W_t
	=(\langle\tilde Y_t,v_1\rangle f_t',\ldots,\langle\tilde  Y_t,v_d\rangle f_t')',\quad 1\le t\le n.
\]
In analogy to $\tilde X_t^d$ and $\tilde{\mathcal X}_n^d(\omega)$ we define now
\[
	\tilde Y_t^d
	=\sum_{k=1}^d\langle\tilde Y_t,v_k\rangle v_k
	\quad\text{and}\quad
	\tilde{\mathcal Y}_n^d(\omega)
	=n^{-1/2}\sum_{t=1}^n\tilde Y_t^de^{-it\omega}.
\]
We have shown that 
\begin{align*}
&\sup_{x\in\mathbb{R}}|P(\max_{1\leq k\leq q}\|\tilde{\mathcal{X}}_n^d(\omega_k)\|^2\leq x)-P(\max_{1\leq k\leq q}\|\tilde{\mathcal{Y}}_n^d(\omega_k)\|^2\leq x)|\\
&\quad\leq \sup_{A\in \mathcal{A}^{\text{sp}}(2d)}|P(S_n^V\in A)-P(S_n^W\in A)|\to 0,\quad\text{when  $n\to\infty$}.
\end{align*}
Therefore, it remains to prove that 
\begin{equation}\label{e:inter}
\lambda_1^{-1}\max_{1\leq j\leq q}\|\tilde{\mathcal{Y}}_n^d(\omega_j)\|^2-b_q^d=\lambda_1^{-1}
\max_{1\leq j\leq q}\sum_{k=1}^d|\langle\mathcal{\tilde Y}_n(\omega_j),v_k\rangle|^2-b_q^d\convd \mathcal G.
\end{equation}
 To this end we introduce $(\tilde{\lambda}_k,\tilde v_k)$, which are the pairs of eigenvalues and eigenfunctions of $\operatorname{Var}(\tilde X_1)$. In a first step we show that
\begin{equation}
\label{e:firststep}
\lambda_1^{-1}\max_{1\leq j\leq q}\sum_{k=1}^d|\langle\mathcal{\tilde Y}_n(\omega_j),\tilde{v}_k\rangle|^2-b_n^d\convd \mathcal G.
\end{equation}
Let us denote
\[
	C_{kj}
	=\tilde\lambda_k^{-1/2}n^{-1/2}\sum_{t=1}^n\langle\tilde Y_t,\tilde v_k\rangle\cos(t\omega_j)
	\quad\text{and}\quad S_{kj}
	=\tilde\lambda_k^{-1/2}n^{-1/2}\sum_{t=1}^n\langle \tilde Y_t,\tilde v_k\rangle\sin(t\omega_j),
\]
with $1\le k\le d$ and $1\le j\le q$. These $2dq$ variables are mutually independent and $N(0,\frac{1}{2})$ distributed. Thus
\begin{equation*}
	\max_{1\le j\le q} \sum_{k=1}^d|\langle\mathcal{\tilde Y}_n(\omega_j),\tilde v_k\rangle|^2
	=\max_{1\le j\le q}\sum_{k=1}^d\tilde\lambda_kE_{kj},
\end{equation*}
where $E_{kj}=C_{kj}^2+S_{kj}^2\stackrel{\text{iid}}{\sim}\operatorname{Exp}(1)$.
Moreover, we have 
\begin{align*}
	\Bigl|\max_{1\le j\le q}\sum_{k=1}^d\tilde\lambda_kE_{kj}
	-\max_{1\le j\le q}\sum_{k=1}^d\lambda_kE_{kj}\Bigr|
	&\le\max_{1\le j\le q}|\sum_{k=1}^d\tilde\lambda_kE_{kj}-\sum_{k=1}^p\lambda_kE_{kj}|\\
	&\le\sum_{k=1}^d|\tilde\lambda_k-\lambda_k|\max_{1\le j\le q}E_{kj}.
\end{align*}
It is a basic result that $\max_{1\le j\le q}E_{kj}=O_P(\log n)$ and \autoref{lemma:convofcovop}  yields $|\tilde\lambda_k-\lambda_k|\le\|\operatorname{Var}(\tilde X_1)-\operatorname{Var}(X_1)\|=o(n^{-(1-2/r)})$ as $n\to\infty$. Hence, combining these results with \autoref{hypofixeddim}, we get \eqref{e:firststep}.

The last step in the proof is to show \eqref{e:inter} and this in turn will follow from \eqref{e:firststep}
if we prove that
\begin{equation}\label{eq:eigenerror}
	\max_{1\le j\le q}\sum_{k=1}^d|\langle\mathcal{\tilde Y}_n(\omega_j),c_kv_k\rangle|^2
	-\max_{1\le j\le q}\sum_{k=1}^d|\langle\mathcal{\tilde Y}_n(\omega_j),\tilde v_k\rangle|^2
	=o_P(1),\quad n\to\infty.
\end{equation}
The absolute of the left-hand side in \eqref{eq:eigenerror} is bounded by
\begin{align}
	&\max_{1\le j\le q}\Bigl|\sum_{k=1}^d\{|\langle\mathcal{\tilde Y}_n(\omega_j),c_kv_k\rangle|^2
	-|\langle\mathcal{\tilde Y}_n(\omega_j),\tilde v_k\rangle|^2\}\Bigr|\notag\\
	&=\max_{1\le j\le q}\Bigl|\sum_{k=1}^d\{|\langle\mathcal{\tilde Y}_n(\omega_j),c_kv_k-\tilde v_k\rangle|^2
	+2\operatorname{Re}[\langle\mathcal{\tilde Y}_n(\omega_j),c_kv_k-\tilde v_k\rangle\langle\tilde v_k,\mathcal{\tilde Y}_n(\omega_j)\rangle]\}\Bigr|\notag\\
	&\le\sum_{k=1}^d \max_{1\le j\le q}|\langle\mathcal{\tilde Y}_n(\omega_j),c_kv_k-\tilde v_k\rangle|^2\label{eq:secterm1}\\
	&\quad+2\sum\max_{1\le j\le q}|\langle\mathcal{\tilde Y}_n(\omega_j),\tilde v_k\rangle|\max_{1\le j\le q}|\langle\mathcal{\tilde Y}_n(\omega_j),c_kv_k-\tilde v_k\rangle|.\label{eq:secterm}
\end{align}
The components of the random vector
\[
	\left(\begin{array}{c}
	n^{-1/2}\sum_{t=1}^n\langle\tilde Y_t,c_kv_k-\tilde v_k\rangle\cos(t\omega_1)\\
	n^{-1/2}\sum_{t=1}^n\langle\tilde Y_t,c_kv_k-\tilde v_k\rangle\sin(t\omega_1)\\
	\ldots\\
	n^{-1/2}\sum_{t=1}^n\langle\tilde Y_t,c_kv_k-\tilde v_k\rangle\cos(t\omega_q)\\
	n^{-1/2}\sum_{t=1}^n\langle\tilde Y_t,c_kv_k-\tilde v_k\rangle\sin(t\omega_q)
	\end{array}\right)
\]
are uncorrelated and the covariance matrix is given by $2^{-1}\operatorname E|\langle\tilde Y_1,c_kv_k-\tilde v_k\rangle|^2I_{2q}$. A summand in \eqref{eq:secterm1} is given by
\begin{align*}
	&\max_{1\le j\le q}\Bigl|n^{-1/2}\sum_{t=1}^n\langle\tilde Y_t,c_kv_k-\tilde v_k\rangle e^{-it\omega_j}\Bigr|^2=\\
	&=\operatorname E|\langle \tilde Y_1,c_kv_k-\tilde v_k\rangle|^2\max_{1\le j\le q}\Bigl|(\operatorname E|\langle \tilde Y_1,c_kv_k-\tilde v_k\rangle|^2)^{-1/2}n^{-1/2}\sum_{t=1}^n\langle\tilde Y_t,c_kv_k-\tilde v_k\rangle e^{-it\omega_j}\Bigr|^2\\
	&\le\operatorname E\|\tilde Y_1	\|^2\|c_kv_k-\tilde v_k\|^2\max_{1\le j\le q}\Bigl|(\operatorname E|\langle \tilde Y_t,c_kv_k-\tilde v_k\rangle|^2)^{-1/2}n^{-1/2}\sum_{t=1}^n\langle\tilde Y_t,c_kv_k-\tilde v_k\rangle e^{-it\omega_j}\Bigr|^2.
\end{align*}
Using \autoref{lem:convrateeigen} in Section~\ref{sec:aux}, $\|c_kv_k-\tilde v_k\|^2=o(n^{-2(1-2/r)})$ as $n\to\infty$ and
\[
	\max_{1\le j\le q}\Bigl|(\operatorname E|\langle \tilde Y_1,c_kv_k-\tilde v_k\rangle|^2)^{-1/2}n^{-1/2}\sum_{t=1}^n\langle\tilde Y_t,c_kv_k-\tilde v_k\rangle e^{-it\omega_j}\Bigr|^2
	=O_P(\log n)
\]
as $n\to\infty$ (since this is the maximum of $q$ iid standard exponential random variables) shows that \eqref{eq:secterm1} tends to 0. Similar arguments show that \eqref{eq:secterm} goes to $0$ in probability. Hence \eqref{eq:eigenerror} holds.
\end{proof}

\begin{proof}[Proof of \autoref{thm:finite}]
The proof is basically identical to the proof of \autoref{mainthm0}. The main difference here is that if we consider the approximating Gaussian process $\{Y_t\}$ with $\operatorname{Var}(Y_1)=\operatorname{Var}(X_1)$  then
\[
	\max_{1\le j\le q}\Bigl\|n^{-1/2}\sum_{t=1}^nY_te^{-it\omega_j}\Bigr\|^2
	=\max_{1\le j\le q}\Bigl\{\sum_{k=1}^dE_{kj}\Bigr\},
\]
where $E_{kj}$ are iid $\operatorname{Exp}(1)$ random variables with $1\le k\le d$ and $1\le j\le q$. Then $\sum_{k=1}^dE_{kj}$ are iid  $\operatorname{Gamma}(d,1)$ random variables. The limiting distribution of the maximum can be found in Example~1 of \citet{kang:serfozo:1999} or Table~3.4.4 of \citet{embrechts:1997}. The proof is complete.
\end{proof}

Now we let $d$ grow to infinity. In this case we need a version of \autoref{p:cher} where the dependence of the constant $C$ on $b$ and $s$ is explicit.  We provide such a result in the following proposition which may be of independent interest. The proof is outlined in Appendix~A.

\begin{Proposition}\label{p:cher+c}
Suppose (i), (ii), (iv) hold with $q\ge4$ hold and $\{B_n\}_{n\ge1}$ is a bounded sequence, then 
\begin{equation}\label{cherrhoexplicit}
	\rho_n(\mathcal{A}^{\mathrm{sp}}(s))
	\le C\cdot \frac{s^4\log^{7/6}(pn)}{b^{1/2}n^{1/6}},
\end{equation}
where $C$ is a constant that does not depend on $n$, $b$, $p$ or $s$. 
\end{Proposition}

For the proofs of  \autoref{mainthm} and \autoref{fancythm} we don't need a truncation argument. We denote the DFT of $Y_1,\ldots,Y_n$ by $\mathcal{Y}_n(\omega)$, $\mathcal{Y}_n^d(\omega)$ is its projection onto $\{v_1,\ldots, v_d\}$ and $\widetilde M_n^{d}=\max_{j=1,\dots,q}\| \mathcal Y_n^d(\omega_j)\|^2$ and $ M_n^{d}=\max_{j=1,\dots,q}\| \mathcal X_n^d(\omega_j)\|^2$.

Now we are ready to prove \autoref{mainthm}.
\begin{proof}[Proof of \autoref{mainthm}]
We have
\begin{align}
	&|P( (M_n^d-b_q^d)/\lambda_1\leq x) - e^{-e^{-x}}|\nonumber\\
	&\le  | P(\widetilde{M}_n^d-b_q^d)/\lambda_1 \leq x ) - e^{-e^{-x}}| +\rho_n(\mathcal A^{\mathrm{sp}}(2d)),\label{rhoFG}
\end{align}
where
\begin{align} \label{defrho}
	\rho_n(\mathcal A^{\mathrm{sp}}(2d))
	=\sup_{x\in \mathbb{R}}| P(M_{n}^{d}\leq x)-P(\tilde{M}_{n}^{d}\leq x)|.
\end{align}
We consider the normalized partial sums
\[
	S_n^V
	=n^{-1/2}\sum_{t=1}^n\xi_t^d\otimes f_t
	=n^{-1/2}\sum_{t=1}^n V_t.
\]
where
\begin{equation}\label{eq:xi_t}
	\xi_t^d
	=(\langle X_t,v_1\rangle, \ldots,\langle  X_t,v_d\rangle)'
\end{equation}
and $f_t$ is defined by \eqref{eq:Wt2q}. Like we showed in the proof of Theorem~\ref{mainthm0} we have that
$$
 	P(M_n^d\le x)=P(\D\in \cap_{k=1}^q A_k).
$$
where $\cap_{k=1}^q A_k$ is  a $2d$-sparsely convex set.

Set $B=\max\{E\| X_1\|^3,(E\| X_1\|^4)^{1/2},(2^{-1}E\|X_1\|^4)^{1/4}\}$. We aim to apply~\eqref{cherrhoexplicit} with $p=2dq$ and $s=2d$. Since $|V_{t,j}|\leq \|X_t\|$ for all $j=1,\dots,2dq$, we see that (ii) and (iv) are satisfied with $B_n=B$ and condition (i) follows from \autoref{lemM1} in Section~\ref{sec:aux} with $b=\lambda_d/2$. Hence, by Proposition~\ref{p:cher+c} we get
\begin{equation}\label{rhopn}
	\rho_n(\mathcal A^{\mathrm{sp}}(2d_n))
	\le C\cdot\frac{ d_n^4 \;  \log^{7/6}(d_nn^2)  }{\lambda_{d_n}^{1/2}\; n^{ 1/6 }}
\end{equation} 
for $n\ge1$, where $C$ is a universal constant.  Under assumption~\eqref{assmain}, the right hand side of \eqref{rhopn} goes to $0$ as $n\to\infty$.

Since $Y_1,Y_2,\ldots$ are Gaussian random elements, $\|\mathcal{Y}^{(d)}_n(\omega_1)\|^2,\ldots,\|\mathcal{Y}^{(d)}_n(\omega_d)\|^2$
are iid variables followin a $\operatorname{Hypo}(\lambda_1^{-1},\ldots,\lambda_d^{-1})$ distribution. \autoref{lemalpha} implies that the first term on the right hand side of \eqref{rhoFG} goes to $0$ as $n\to\infty$ under assumptions \eqref{assLem5} and \eqref{gamma0}. The proof is complete.
\end{proof}

\begin{proof}[Proof of \autoref{fancythm}] We start by noting that
\begin{equation*}
\lambda_1^{-1}(M_n  - b_q) = \lambda_1^{-1}(M_n - M_n^{d_n})+  \lambda_1^{-1}(M_n^{d_n} - b_q^{d_n})   + \lambda_1^{-1} (b_n^{d_n} - b_q).
\end{equation*}
The third term converges to zero by \autoref{munp}, and the second term converges in distribution to a Gumbel random variable under the assumptions of \autoref{mainthm}. Hence, by Slutsky's theorem, the convergence in distribution of $\lambda_1^{-1}(M_n  - b_q) $ to the standard Gumbel distribution holds if we verify that the first term tends to zero in probability. 

To this end, we define
\[
	\delta_j
	=\| \mathcal{X}_n(\omega_j)\|^2-\| \mathcal{X}_n^{d}(\omega_j)\|^2
	=\sum_{k>d}  \frac{1}{n}\Bigl| \sum_{t=1}^n  \langle X_t,v_k\rangle e^{-\ii t\omega_j}\Bigr|^2.
\]
For any $a>0$, we have
\begin{align}
\notag P( |M_n-M_n^{d} | >a ) &= P( M_n-M_n^{d}>a ) \\
\notag &=  P\Bigl( \max_{j=1,\dots,q} \Bigl\lbrace \| \mathcal{X}^{d}(\omega_j)\|^2+\delta_j\Bigr\rbrace -M_n^{d} >a \Bigr)  \\ 
\notag &\leq   P\Bigl( \max_{j=1,\dots,q} \delta_j >a\Bigr) \\ 
\notag &\leq  \sum_{j=1}^q \sum_{k>d}  P\Bigl( \frac{1}{n}\Bigl| \sum_{t=1}^n  \langle X_t,v_k\rangle \cos(t\omega_j) \Bigr|^2  > a\ell_k /2 \Bigr)\\ 
 &\qquad+ \sum_{j=1}^q \sum_{k>d}P\Bigl( \frac{1}{n}\Bigl| \sum_{t=1}^n  \langle X_t,v_k\rangle \sin(t\omega_j) \Bigr|^2 > a\ell_k/2  \Bigr)\label{eq:sinterm}.
\end{align}

Since $\{\langle X_t,v_k\rangle \cos(t\omega_j)\}_{1\le t\le n}$ are independent random variables with zero means and $\operatorname E|\langle X_t,v_k\rangle \cos(t\omega_j)|^r\le\operatorname E\|X_1\|^r<\infty$ for some $r>2$, Markov's inequality and Rosenthal's inequality (see \citet{rosenthal:1970}) lead to
\begin{align*}
&P\Bigl( \frac{1}{n}\Bigl| \sum_{t=1}^n  \langle X_t,v_k\rangle \cos(t\omega_j) \Bigr|^2  > a\ell_k /2 \Bigr)\le\\
&\quad\leq C_r \, (n a\ell_k /2)^{-r/2}\Big[\sum_{t=1}^nE|  \langle X_t,v_k\rangle|^r + \Bigl(\sum_{t=1}^n E|\langle X_t,v_k\rangle|^2\Bigr)^{r/2}\Bigr] \\
&\quad \leq C_r  \,   (n a\ell_k /2)^{-r/2} [n E|  \langle X_1,v_k\rangle|^r+  ( n\lambda_k )^{r/2}]  \\
 &  \quad \leq C_r (2/a)^{r/2}[  n^{1-r/2}\ell_k^{-r/2} E|  \langle X_1,v_k\rangle|^r  + (\lambda_k/\ell_k)^{r/2}],
\end{align*}
where $C_r$ is a constant depending only on $r$. \eqref{eq:sinterm} can be bounded in an analogous way and summation over $j$ and $k$ gives conditions \eqref{uglygood} and \eqref{condlemmMMb}. The proof is complete.
\end{proof}

\subsection{Domain of attraction of Gumbel distribution}\label{sec:Gum}
First, we show that, for fixed $d\ge1$, the hypoexponential distribution with strictly increasing parameters belongs to the domain of attraction of the Gumbel distribution.
\begin{lem}\label{hypofixeddim}
Let $d\ge1$ be fixed. Suppose that $\xi_1,\ldots,\xi_n$ are iid $\operatorname{Hypo}(\lambda_1^{-1},\ldots\lambda_d^{-1})$ random variables with $\lambda_k>\lambda_{k+1}$ for all $1\le k\le d$. Then
\begin{equation}\label{eq:convtogum}
	\lambda_1^{-1}(\max\{\xi_1,\ldots,\xi_n\}-b_n^d)\convd\mathcal G\quad\text{as}\quad n\to\infty,
\end{equation}
where $b_n^d$ is given by \eqref{defseq}.
\end{lem}
\begin{proof}[Proof of \autoref{hypofixeddim}]
Since $\lambda_k>\lambda_{k+1}$ for all $1\le k\le d$, the cdf of  $\operatorname{Hypo}(\lambda_1^{-1},\ldots\lambda_d^{-1})$ is given by 
\begin{equation}\label{eq:cdfofhypo}
	F^{(d)}(x)=\sum_{k=1}^d\alpha_{k,d}F_k(x)
\end{equation}
for $x\in\mathbb R$, where $\alpha_{k,d}=\prod_{j=1,j\neq k}^d(1-\lambda_j/\lambda_k)^{-1}$ and $F_k$ is the cdf of $\operatorname{Exp}(\lambda_k^{-1})$ for $k\ge1$. Hence, Theorem~2 of \citet{kang:serfozo:1999} implies that $\max\{\xi_1,\ldots,\xi_n\}$ is asymptotically distributed as the standard Gumbel distribution with the normalizing constants given by \eqref{defseq}.
\end{proof}
Under certain assumptions on the parameters of the hypoexponential distribution and the the growth rate of $d=d_n$, we show that convergence \eqref{eq:convtogum} still holds even if $d=d_n\to\infty$ as $n\to\infty$.

\begin{lem} \label{lemalpha} 
Suppose that condition \eqref{assLem5} is satisfied and that $d=d_n = O(n^{\gamma_0})$ as $n\to\infty$, with $\gamma_0$ satisfying \eqref{gamma0}. Then convergence \eqref{eq:convtogum} holds.
\end{lem}

\begin{proof}[Proof of \autoref{lemalpha}]
Fix $x\in\mathbb R$. Using \eqref{eq:cdfofhypo} and the fact that $\sum_{k=1}^d\alpha_{k,d}=1$, we obtain
\begin{align}
	P(\lambda_1^{-1}(\max\{\xi_1,\ldots,\xi_n\}-b_n^d) \leq x)
	\notag&=F^{(d)}\big(\lambda_1 x + b_n^d \big)^n\\
	\label{Finf}&=\Bigl[1-\frac{e^{-x}}n-\sum_{k=2}^d\alpha_{k,d}\Big(\frac{e^{-x}}{n\alpha_{1,d}}\Bigr)^{\lambda_1/\lambda_k}\Bigr]^n.
\end{align}
We need to show that
\begin{equation}\label{impterm}
	\sum_{k=2}^d\alpha_{k,d}\Big(\frac{e^{-x}}{n\alpha_{1,d}}\Bigr)^{\lambda_1/\lambda_k}=o(n^{-1})
\end{equation}
as $n\to\infty$, which implies that \eqref{Finf} converges to the Gumbel distribution function.

Denote 
\[
a_{k,n}:=n\alpha_{k,d}\Big(\frac{e^{-x}}{n\alpha_{1,d}}\Bigr)^{\lambda_1/\lambda_k}
\]
and let $d_n = O(n^{\gamma_0})$ as $n\to\infty$ for some $0<\gamma_0<1$.
We first remark that
\begin{equation} \label{alpha1}
\frac{1}{\alpha_{1,d}} =   \prod_{j=2}^{d}(1-\lambda_j/\lambda_1) \leq (1-\lambda_d/\lambda_1)^{d-1} \leq 1.
\end{equation}
Observe that condition \eqref{assLem5} is equivalent to the following condition: there exists $k_0\ge1$ such that
\begin{equation}\label{eq:ratio}
	\lambda_j/\lambda_k\ge k/j
\end{equation}
for each $k\ge k_0$ and $1\le j\le k-1$. Let $k_1\geq  k_0$ be such that $d_n(d_n/n)^{k_1 -1}\to 0$, which is possible since $\gamma_0<1$. Denote
\begin{equation} \label{splitAn}
\underline{A}_n =  \sum_{k=2}^{k_1-1} a_{k,n}  \quad \text{ and }\quad \overline{A}_n =  \sum_{k=k_1}^{d_n}a_{k,n}.
\end{equation}

For $k\geq k_1$, using \eqref{eq:ratio} and the fact that $\lambda_j/\lambda_k\le k/j$ for $j>k$,
\begin{equation} \label{boundalphakp1}
|\alpha_{k,d}|
 \leq \prod_{j=1}^{k-1}\frac{j}{k-j}   \prod_{j=k+1}^{d}\frac{j}{j-k}  =    {d\choose k}\leq \biggl(\frac{e d}{k}\biggr)^k.
\end{equation}
Choose $n_0\ge1$ such that $e^{-x}/n\le 1$ for $n\ge n_0$. For $n\ge n_0$, using \eqref{alpha1}, \eqref{eq:ratio} and \eqref{boundalphakp1}, we obtain
\begin{align*}
	 |\overline{A}_n|
	 &\leq\sum_{k=k_1}^{d_n}n\Bigl(\frac{e d_n }{k}\Bigr)^k\Bigl(\frac{e^{-x}}{n}\Bigr)^{\lambda_1/\lambda_k }\\
	 &\leq  d_n\sum_{k=k_1}^{d_n} \Bigl(\frac{e^{1-x}}{k}\Bigr)^k \Bigl(\frac{d_n}{n}\Bigr)^{k -1}\\
	  &\leq d_n\Bigl(\frac{d_n}{n}\Bigr)^{k_1 -1} \sum_{k=1}^{\infty} \Bigl(\frac{e^{1-x}}{k}\Bigr)^k,
\end{align*}
where $d_n(d_n/n)^{k_1 -1}\to 0$ as $n\to\infty$.

Next, for $k<k_1$, set $\nu_k=\prod_{j=1,j\neq k}^{k_1-1}|1-\lambda_j/\lambda_k|^{-1}$. Since $\lambda_j/\lambda_k\le k/j$ for $j\ge k_1$ and $k<k_1$ using \eqref{assLem5}, we obtain
\begin{align}
	|\alpha_{k,d}|
	 \notag&=  \nu_k\prod_{j=k_1}^{d}(1-\lambda_j/\lambda_k)^{-1}\\
	  \notag&\leq   \nu_k\prod_{j=k_1}^{d}\frac{j}{j-k}\\
	   \notag&= \nu_k \cdot  \frac{ d! (k_1-k-1)!}{(k_1-1)! (d-k)!}\\
	\label{eq:secondbound}&\leq \nu_k \cdot d^k \cdot  \frac{  (k_1-k-1)!}{(k_1-1)!}.
\end{align}
Thus, using \eqref{alpha1} and \eqref{eq:secondbound},
\begin{equation}\label{eq:finalterm}
|\underline{A}_n|  \leq  \sum_{k=2}^{k_1} \nu_k \cdot d^k_n \,  \cdot  \frac{  (k_1-k-1)!}{(k_1-1)!}\cdot  e^{-x\lambda_1/\lambda_k}\biggl(\frac1n\biggr)^{\lambda_1/\lambda_k -1}  =   O\Big(\max_{2\leq k\leq k_1}\Bigl\{ \frac{ d_n^k}{n^{\lambda_1/\lambda_k -1}}\Bigr\} \Bigr)
\end{equation}
as $n\to\infty$. If $\gamma_0<\min_{2\leq k\leq k_1}\{k^{-1}(\lambda_1/\lambda_k -1)\}$, \eqref{eq:finalterm} tends to $0$ as $n\to\infty$. The proof is complete.
\end{proof}

\subsection{Linear processes}\label{sec:lp}
The method for transferring the iid setting to linear processes is similar as for the central limit theorem and the functional central limit theorem for linear processes under the absolute summability of $a_k$'s (see e.g.\ \citet{merleved:peligrad:utev:1997} and \citet{rackauskas:suquet:2010}).
\begin{lem}\label{lem:boundtransfer}
Suppose that $\{X_t\}_{t\in\mathbb Z}$ is a linear process defined by \eqref{eq:linproc} such that $\sum_{k=-\infty}^\infty\|a_k\|<\infty$. Then for $n\geq 3$ we have
\begin{equation}\label{e:lin1}
	\operatorname E\max_{1\le j\le q}\Bigl\|\sum_{t=1}^nX_te^{-it\omega_j}\Bigr\|^2
	\le\Bigl(\sum_{k=-\infty}^\infty\|a_k\|\Bigr)^2\operatorname E\max_{1\le j\le q}\Bigl\|\sum_{t=1}^n\varepsilon_te^{-it\omega_j}\Bigr\|^2.
\end{equation}
\end{lem}
\begin{proof}
The left-hand side of \eqref{e:lin1} is given as
\begin{align*}
	&\operatorname E\max_{1\le j\le q}\Bigl\|\sum_{k=-\infty}^\infty a_k\Bigl(\sum_{s=1-k}^{n-k}\varepsilon_se^{-is\omega_j}\Bigr)e^{-ik\omega_j}\Bigr\|^2\le\\
	&\le\operatorname E\max_{1\le j\le q}\Bigl(\sum_{k=-\infty}^\infty\Bigl\|a_k\Bigl(\sum_{s=1-k}^{n-k}\varepsilon_se^{-is\omega_j}\Bigr)e^{-ik\omega_j}\Bigr\|\Bigr)^2\\
	&\le\operatorname E\max_{1\le j\le q}\Bigl(\sum_{k=-\infty}^\infty\|a_k\|\Bigl\|\sum_{s=1-k}^{n-k}\varepsilon_se^{-is\omega_j}\Bigr\|\Bigr)^2.
\end{align*}
The proof is complete.
\end{proof}

\begin{proof}[Proof of \autoref{lem:lp}]
Some little algebra shows that
\begin{align*}
	\mathcal X_n(\omega)-A(\omega)\mathcal{E}_n(\omega)
	&=n^{-1/2}\sum_{k=-\infty}^\infty a_k\Bigl(\sum_{s=1-k}^{n-k}\varepsilon_se^{-is\omega}-\sum_{t=1}^n\varepsilon_te^{-it\omega}\Bigr)e^{-ik\omega}\\
	&=n^{-1/2}\sum_{k=-\infty}^\infty a_k(\Delta_{nk}(\omega))e^{-ik\omega}.
\end{align*}
Since $\Delta_{nk}(\omega)$ has $2(|k|\wedge n)$ summands, it follows from \autoref{thm:Ologn}
that
\[
	\operatorname E\max_{1\le j\le q}\|\Delta_{nk}(\omega_j)\|
	\le(\operatorname E\max_{1\le j\le q}\|\Delta_{nk}(\omega_j)\|^2)^{1/2}
	\le(2C\{|k|\log(2|k|)\wedge n\log(2n)\})^{1/2},
\]
where $C>0$ is some constant. 
Hence,
\begin{align*}
	&\operatorname E\max_{1\le j\le q}\|\mathcal X_n(\omega_j)-A(\omega_j)\mathcal E_n(\omega_j)\|\le\\
	&\le n^{-1/2}\sum_{k=-\infty}^\infty\|a_k\|\operatorname E\max_{1\le j\le q}\|\Delta_{nk}(\omega_j)\|\\  	&\ll n^{-1/2}\Bigl[\sum_{|k|<n\atop k\neq 0}(|k|\log(2|k|))^{1/2}\|a_k\|	+(n\log(2n))^{1/2}\sum_{|k|\ge n}\|a_k\|\Bigr]\\
	&\ll \log^{-1/2}(2n)\Bigl[\sum_{|k|<n\atop k\neq 0}\Bigl[\frac{|k|}n\frac{\log(2n)}{\log(2|k|)}\Bigr]^{1/2}\log(2|k|)\|a_k\|+\sum_{|k|\ge n}\log(2|k|)\|a_k\|\Bigr].
\end{align*}
It follows that $\operatorname E\max_{1\le j\le q}\|\mathcal X_n(\omega_j)-A(\omega_j)\mathcal E_n(\omega_j)\|=o(\log^{-1/2}n)$ as $n\to\infty$, which is a sufficient condition for \eqref{eq:lp_dft}. Since
\begin{align*}
	&\mathcal X_n(\omega)\otimes\mathcal X_n(\omega)
	-A(\omega)\mathcal E_n(\omega)\otimes A(\omega)\mathcal E_n(\omega)=\\
	&=(\mathcal X_n(\omega)-A(\omega)\mathcal E_n(\omega))\otimes(\mathcal X_n(\omega)-A(\omega)\mathcal E_n(\omega))\\
	&+(\mathcal X_n(\omega)-A(\omega)\mathcal E_n(\omega))\otimes A(\omega)\mathcal E_n(\omega)\\
	&+A(\omega)\mathcal E_n(\omega)\otimes(\mathcal X_n(\omega)-A(\omega)\mathcal E_n(\omega)),
\end{align*}
for $\omega\in[-\pi,\pi]$ and $\|x\otimes y\|=\|x\|\|y\|$ for $x,y\in H$, we obtain
\begin{align}
	&\max_{1\le j\le q}\|\mathcal X_n(\omega_j)\otimes\mathcal X_n(\omega_j)
	-A(\omega_j)\mathcal E_n(\omega_j)\otimes A(\omega_j)\mathcal E_n(\omega_j)\|_{\mathcal S}\notag\\
	&\quad\le\max_{1\le j\le q}\|\mathcal X_n(\omega_j)-A(\omega_j)\mathcal E_n(\omega_j)\|^2\notag\\
	&\qquad+2\max_{1\le j\le q}\|A(\omega_j)\mathcal E_n(\omega_j)\|\max_{1\le j\le q}\|\mathcal X_n(\omega_j)-A(\omega_j)\mathcal E_n(\omega_j)\|.\label{eq:lp_bound}
\end{align}
Also,
\begin{align}
	\operatorname E\max_{1\le j\le q}\|A(\omega_j)\mathcal E_n(\omega_j)\|
	&\le\max_{1\le j\le q}\|A(\omega_j)\|(\operatorname E\max_{1\le j\le q}\|\mathcal E_n(\omega_j)\|^2)^{1/2}\notag\\
	&\le\max_{1\le j\le q}\|A(\omega_j)\|(C\log n)^{1/2}.\label{eq:lp_max}
\end{align}
Then \eqref{eq:lp_bound} together with \eqref{eq:lp_dft} and \eqref{eq:lp_max} implies \eqref{eq:lp_pdg}. The proof is complete.
\end{proof}

\begin{proof}[Proof of \autoref{lem:lpfilter}]
We have that
\begin{align*}
	&\Bigl|\max_{1\le j\le q}\|A^{-1}(\omega_j)\mathcal X_n(\omega_j)\|^2-\max_{1\le j\le q}\|\mathcal E_n(\omega_j)\|^2\Bigr|\\
	&\le\max_{1\le j\le q}\bigl|\|A^{-1}(\omega_j)\mathcal X_n(\omega_j)\|^2-\|\mathcal E_n(\omega_j)\|^2\bigr|\\
	&=\max_{1\le j\le q}\{
	(
	\|A^{-1}(\omega_j)\mathcal X_n(\omega_j)\|-\|\mathcal E_n(\omega_j)\|
	)(\|A^{-1}(\omega_j)\mathcal X_n(\omega_j)\|+\|\mathcal E_n(\omega_j)\|)
	\}\\
	&\le\max_{1\le j\le q}\|A^{-1}(\omega_j)\mathcal X_n(\omega_j)-\mathcal E_n(\omega_j)\|\Bigl(\max_{1\le j\le q}\Bigl\|A^{-1}(\omega_j)\mathcal X_n(\omega_j)\|+\max_{1\le j\le q}\|\mathcal E_n(\omega_j)\|\Bigr)\\
	&\le\sup_{\omega\in[0,\pi]}\|A^{-1}(\omega)\|\max_{1\le j\le q}\|\mathcal X_n(\omega_j)-A(\omega_j)\mathcal E_n(\omega_j)\|\\
	&\qquad\times\Bigl(\ \sup_{\omega\in[0,\pi]}\|A^{-1}(\omega)\|\max_{1\le j\le q}\|\mathcal X_n(\omega_j)\|+\max_{1\le j\le q}\|\mathcal E_n(\omega_j)\|\Bigr).
\end{align*}
According to \autoref{lem:lp},
\[
	\max_{1\le j\le q}\|\mathcal X_n(\omega_j)-A(\omega_j)\mathcal E_n(\omega_j)\|
	=o_p(\log^{-1/2}n)
\]
as $n\to\infty$. Also, it follows from \autoref{lem:boundtransfer} and \autoref{thm:Ologn} that
\[
	\max_{1\le j\le q}\|\mathcal E_n(\omega_j)\|=O_p(\log^{1/2}n)
	\quad\text{and}\quad
	\max_{1\le j\le q}\|\mathcal X_n(\omega_j)\|=O_p(\log^{1/2}n)
\]
as $n\to\infty$, which completes the proof.
\end{proof}

\begin{proof}[Proof of \autoref{thm:mv}]
Denote $Z_t=\Sigma^{-1/2}\varepsilon_t$ with $t=1,\ldots,n$ so that $\mathcal Z_n(\omega)=\Sigma^{-1/2}\mathcal E_n(\omega)$ for $\omega\in[-\pi,\pi]$. Similarly as in the proof of \autoref{lem:lpfilter},
\begin{align*}
	&\Bigl|\max_{1\le j\le q}\|B^{-1}(\omega_j)\mathcal X_n(\omega_j)\|^2-\max_{1\le j\le q}\|\mathcal Z_n(\omega_j)\|^2\Bigr|\\
	&\le\|\Sigma^{-1/2}\|\sup_{\omega\in[0,\pi]}\|A^{-1}(\omega)\|\max_{1\le j\le q}\|\mathcal X_n(\omega_j)-A(\omega_j)\mathcal X_n(\omega_j)\|\\
	&\qquad\times\Bigl(\|\Sigma^{-1/2}\|\sup_{\omega\in[0,\pi]}\|A^{-1}(\omega)\|\max_{1\le j\le q}\|\mathcal X_n(\omega_j)\|+\max_{1\le j\le q}\|\mathcal Z_n(\omega_j)\|\Bigr).
\end{align*}
Using \autoref{lem:lp}, \autoref{lem:boundtransfer} and \autoref{thm:Ologn},
\[
	\Bigl|\max_{1\le j\le q}\|B^{-1}(\omega_j)\mathcal X_n(\omega_j)\|^2-\max_{1\le j\le q}\|\mathcal Z_n(\omega_j)\|^2\Bigr|
	=o_p(1),
\]
as $n\to\infty$ and we use \autoref{thm:finite} to conclude.
\end{proof}

\subsection{Application to FAR(1) models}\label{s:FAR1}
The proof of \autoref{th:appl} is a simple consequence of the following three lemmas. We remark that we can work with $\mathcal{X}_n(\omega_j)$ instead of $\mathcal{Y}_n(\omega_j)$, as those quantities are identical under $H_0$ for any $j=1,\ldots,q$.
\begin{lem}\label{l:appl1}
Under Assumption~\ref{ass:rho} we have
$$
\Bigl|\max_{1\le j\le q}\|(I-e^{-\ii\omega_j}\widehat\rho\,)\mathcal X_n(\omega_j)\|^2-\max_{1\le j\le q}\|(I-e^{-\ii\omega_j}\rho\,)\mathcal X_n(\omega_j)\|^2\Bigr|=o_P\Bigl(\frac{\log n}{a_n}\Bigr).
$$
\end{lem}
\begin{proof}
Let $$v=(I-e^{-\ii\omega_j}\rho\,)\mathcal X_n(\omega_j)\quad\text{and} \quad h=e^{-\ii\omega_j}(\rho-\widehat\rho\,)\mathcal X_n(\omega_j).$$
Then, using $\|v+h\|^2-\|v\|^2=\|h\|^2+\langle v,h\rangle+\overline{\langle v,h\rangle}$ and thus
$$
|\|v+h\|^2-\|v\|^2| \leq \|h\|^2+2\|h\|\|v\|
$$ we get
\begin{align*}
&\Bigl|\max_{1\le j\le q}\|(I-e^{-\ii\omega_j}\widehat\rho\,)\mathcal X_n(\omega_j)\|^2-\max_{1\le j\le q}\|(I-e^{-\ii\omega_j}\rho)\mathcal X_n(\omega_j)\|^2\Bigr|\\
&\quad\leq \max_{1\le j\le q}\bigl|\|(I-e^{-\ii\omega_j}\widehat\rho\,)\mathcal X_n(\omega_j)\|^2-\|(I-e^{-\ii\omega_j}\rho)\mathcal X_n(\omega_j)\|^2\bigr|\\
&\quad\leq (\|\widehat \rho-\rho\|^2+2(1+\|\rho\|)\|\widehat\rho-\rho\|)\max_{1\le j\le q}\|\mathcal X_n(\omega_j)\|^2.
\end{align*}
By \eqref{e:diffdft} $\max_{1\le j\le q}\|\mathcal X_n(\omega_j)\|=O_P(\max_{1\le j\le q}\|\mathcal E_n(\omega_j)\|)$, which in turn is $O_P(\log n)$ by \autoref{fancythm}.
\end{proof}

\begin{lem}\label{l:appl2}
Under Assumption~\ref{ass:rho} we have
$$
\max_{j\geq 1}|\lambda_j-\hat\lambda_j|=o_P\Bigl(\frac{1}{a_n}\Bigr).
$$
\end{lem}
\begin{proof}
By Weyl's lemma it suffices to show that
$\big\|\frac{1}{n}\sum_{t=1}^n \hat\varepsilon_t\otimes\hat\varepsilon_t-\Sigma\big\|=o_P(a_n^{-1})$. (For the sake of simplicity take averages from 1 to $n$.) Since we require 4 moments for the $\varepsilon_t$ it follows that $\big\|\frac{1}{n}\sum_{t=1}^n \varepsilon_t\otimes\varepsilon_t-\Sigma\big\|=O_P(n^{-1/2})$. Hence we have that
\begin{align*}
&\big\|\frac{1}{n}\sum_{t=1}^n \hat\varepsilon_t\otimes\hat\varepsilon_t-\Sigma\big\| \leq \frac{1}{n}\sum_{t=1}^n \big\|\hat\varepsilon_t\otimes\hat\varepsilon_t-\varepsilon_t\otimes\varepsilon_t\big\|+\big\|\frac{1}{n}\sum_{t=1}^n \varepsilon_t\otimes\varepsilon_t-\Sigma\big\|\\
&\quad\leq \frac{1}{n}\sum_{t=1}^n \|\hat\varepsilon_t\otimes\hat\varepsilon_t-\varepsilon_t\otimes\varepsilon_t\big\|+O_P(n^{-1/2})\\
&\quad\leq \frac{2}{n}\sum_{t=1}^n \|\varepsilon_t\|\|\hat\varepsilon_t-\varepsilon_t\|+\frac{1}{n}\sum_{t=1}^n\|\hat\varepsilon_t-\varepsilon_t\|^2+O_P(n^{-1/2})\\
&\quad\leq 2\Bigl(\frac{1}{n}\sum_{t=1}^n \|\varepsilon_t\|^2\Bigr)^{1/2}\Bigl(\frac{1}{n}\sum_{t=1}^n \|\hat\varepsilon_t-\varepsilon_t\|^2\Bigr)^{1/2}+\frac{1}{n}\sum_{t=1}^n\|\hat\varepsilon_t-\varepsilon_t\|^2+O_P(n^{-1/2}).
\end{align*}
Now we have 
$$
\frac{1}{n}\sum_{t=1}^n\|\hat\varepsilon_t-\varepsilon_t\|^2=\frac{1}{n}\sum_{t=1}^n \|(\widehat\rho -\rho)X_{t-1}\|^2\leq \|\widehat\rho -\rho\|^2\times\frac{1}{n}\sum_{t=1}^n \| X_{t-1}\|^2=o_P(a_n^{-2}).
$$
Here we used that by the ergodic theorem $\frac{1}{n}\sum_{t=1}^n \|X_{t-1}\|^2=O_P(1)$ and by the law of large numbers $\frac{1}{n}\sum_{t=1}^n\|\varepsilon_t\|^2=O_P(1)$. Hence the claim follows.
\end{proof}

\begin{lem}\label{l:appl3}
Under \autoref{ass:eigen} and \autoref{ass:rho} we have $$\sum_{j\geq 2}\log(1-\lambda_j/\lambda_1)-\sum_{j=2}^{a_n}\log(1-\hat\lambda_j/\hat\lambda_1)=o_P(1).$$
\end{lem}
\begin{proof}
Since the series $\sum_{j\geq 2}\log(1-\lambda_j/\lambda_1)$ is convergent and $a_n\to\infty$, it suffices to show that
$$\sum_{j\geq 2}^{a_n}\bigl(\log(1-\lambda_j/\lambda_1)-\log(1-\hat\lambda_j/\hat\lambda_1)\bigr)=o_P(1).$$
To this end we note that by the mean value theorem and the monotonicity of $\log'x=x^{-1}$ we have
\begin{align*}
&|\log(1-\lambda_j/\lambda_1)-\log(1-\hat\lambda_j/\hat\lambda_1)|\\
&\quad\leq |\hat\lambda_j/\hat\lambda_1-\lambda_j/\lambda_1|\times\max\Bigl\{\frac{\lambda_1}{\lambda_1-\lambda_j},\frac{\hat\lambda_1}{\hat\lambda_1-\hat\lambda_j}\Bigr\}\\
&\quad\leq |\hat\lambda_j/\hat\lambda_1-\lambda_j/\lambda_1|\times\max\Bigl\{\frac{\lambda_1}{\lambda_1-\lambda_2},\frac{\hat\lambda_1}{\hat\lambda_1-\hat\lambda_2}\Bigr\}.
\end{align*}
By \autoref{l:appl2} we have $\max\{\frac{\lambda_1}{\lambda_1-\lambda_2},\frac{\hat\lambda_1}{\hat\lambda_1-\hat\lambda_2}\}=O_P(1)$ and
\begin{align*}
\max_{j\geq 2}|\hat\lambda_j/\hat\lambda_1-\lambda_j/\lambda_1|&\leq \max_{j\geq 2}|\hat\lambda_j/\hat\lambda_1-\lambda_j/\hat\lambda_1|+\max_{j\geq 2}|\lambda_j/\hat\lambda_1-\lambda_j/\lambda_1|\\
&\leq \frac{1}{\hat\lambda_1}\max_{j\geq 2}(|\hat\lambda_j-\lambda_j|+|\hat\lambda_1-\lambda_1|)\\
&\leq \frac{2}{\hat\lambda_1}\max_{j\geq 1}|\lambda_j-\hat\lambda_j|=o_P(a_n^{-1}).
\end{align*}

\end{proof}

\begin{proof}[Proof of \autoref{th:applconsist}]
Define $N=\lfloor n/d\rfloor$, $r=n-dN\in \{0,\ldots, d-1\}$ and set $\hat\omega=2\pi N/n$. Clearly, when $r=0$, then $\hat\omega=2\pi/d$ and then $\mathcal{S}_n(\hat\omega)=\sqrt{\frac{N}{d}}\sum_{t=1}^d s(t)e^{-\ii \frac{2\pi t}{d}}$. Let us elaborate the term when $r\neq 0$.
$$
\mathcal{S}_n(\hat\omega):=\frac{1}{\sqrt{n}}\sum_{t=1}^n s(t) e^{-\ii \hat\omega t}.
$$
By the $d$-periodicity of $s(t)$ and letting $R_n=\frac{1}{\sqrt{n}}\sum_{t=n-r+1}^n s(t)e^{-\ii \hat\omega t}$ we get
\begin{align*}
\mathcal{S}_n(\hat\omega)&=\frac{1}{\sqrt{n}}\sum_{t=1}^d s(t)\sum_{m=0}^{N-1} e^{-\ii\frac{2\pi(t+md)N}{n}}+R_n\\
&=\sum_{t=1}^d s(t) e^{-\ii \hat\omega t}\times \frac{1}{\sqrt{n}}\sum_{m=0}^{N-1} e^{-\ii\frac{2\pi m(n-r)}{n}}+R_n\\
&=\sum_{t=1}^d s(t) e^{-\ii \hat\omega t}\times \frac{1}{\sqrt{n}}\sum_{m=0}^{N-1} e^{\ii\frac{2\pi rm}{n}}+R_n.
\end{align*}
Now using the formula $\sum_{m=0}^{N-1} e^{\ii mx}=\frac{\sin(Nx/2)}{\sin(x/2)}e^{\ii x\frac{N-1}{2}}$ we have with  $x=2\pi r/n$
$$
\Bigl|\sum_{m=0}^{N-1} e^{\ii\frac{2\pi rm}{n}}\Bigr|=\Bigl|\frac{\sin(\frac{N\pi r}{n})}{\sin(\frac{\pi r}{n})}\Bigr|\geq \frac{\sin(\pi\min\{\frac{d-r}{d},\frac{Nr}{n}\})}{\sin(\frac{\pi r}{n})}.
$$
For the last inequality we use that $\pi\min\{\frac{d-r}{d},\frac{Nr}{n}\}\in [0,\pi/2]$ and that $\sin(x)$ is increasing in this interval. Recall, moreover, that $x/2\leq \sin(x)\leq x$ for $x \in [0,\pi/2]$. Hence $d/n\to 0$ implies that for large enough $n$ 
$$
\frac{\sin(\pi\min\{\frac{d-r}{d},\frac{Nr}{n}\})}{\sin(\frac{\pi r}{n})}\geq \frac{n}{2 r}\min\Bigl\{\frac{d-r}{d},\frac{Nr}{n}\Bigr\}\geq \frac{N}{2d}.
$$

Because of $\|R_n\|=O(n^{-1/2})$ we can conclude that for  $n\geq n_0$ we have
$$
\|\mathcal{S}_n(\hat\omega)\|\geq \frac{\sqrt{N}}{3d^{3/2}}\Bigl\|\sum_{t=1}^d s(t) e^{-\ii \hat\omega t}\Bigr\|.
$$
We note, moreover, that
$$
\Bigl\|\sum_{t=1}^d s(t) \Bigl(e^{-\ii \hat\omega t}-e^{-\ii \frac{2\pi}{d} t}\Bigr)\Bigr\|\leq \sum_{t=1}^d \|s(t)\|\Bigl|\hat\omega-\frac{2\pi}{d}\Bigr|t\leq \sum_{t=1}^d \|s(t)\|\Bigl|\frac{2\pi r}{n}\Bigr|=O(d/N).
$$
Because of \eqref{e:applcons} we may hence conclude that for a large enough $n$
$$
\|\mathcal{S}_n(\hat\omega)\|\geq \frac{\sqrt{n}}{4d^{2}}\Bigl\|\sum_{t=1}^d s(t) e^{-\ii \frac{2\pi t}{d}}\Bigr\|,
$$
and thus
\begin{align*}
&\max_{1\leq j\leq q}\|(I-e^{-\ii\omega_j}\widehat\rho\,)\mathcal Y_n(\omega_j)\|=\max_{1\leq j\leq q}\|(I-e^{-\ii\omega_j}\widehat\rho\,)\mathcal (S_n(\omega_j)+\mathcal X_n(\omega_j))\|\\
&\quad\geq \|(I-e^{-\ii\hat\omega}\widehat\rho\,)\mathcal S_n(\hat\omega)\|-
\max_{1\leq j\leq q}\|(I-e^{-\ii\omega_j}\widehat\rho\,)\mathcal X_n(\omega_j)\|\\
&\quad\geq (1-\|\widehat\rho\,\|)\frac{\psi_n\sqrt{\log n}}{4}- \max_{1\leq j\leq q}\|(I-e^{-\ii\omega_j}\widehat\rho\,)\mathcal X_n(\omega_j)\|.
\end{align*}
Then we have
\begin{align*}
\max_{1\leq j\leq q}\|(I-e^{-\ii\omega_j}\widehat\rho\,)\mathcal X_n(\omega_j)\|&=\max_{1\leq j\leq q}\|(I-e^{-\ii\omega_j}\widehat\rho\,)(I-e^{-\ii\omega_j}\rho)^{-1}(I-e^{-\ii\omega_j}\rho)\mathcal X_n(\omega_j)\|\\
&\leq \frac{1+\|\widehat\rho\,\|}{1-\|\rho\|} \max_{1\leq j\leq q}\|\mathcal X_n(\omega_j)\|=O_P(\sqrt{\log n}).
\end{align*}
This last bound is obtained by $\|\widehat\rho-\rho^\prime\|\convP 0$, \autoref{lem:lpfilter} and \autoref{fancythm}.
So by \eqref{e:applcons} it follows that $\max_{1\leq j\leq q}\|(I-e^{-\ii\omega_j}\widehat\rho\,)\mathcal X_n(\omega_j)\|=o_P(\psi_n\sqrt{\log n})$. This shows that $\max_{1\leq j\leq q}\|(I-e^{-\ii\omega_j}\widehat\rho\,)\mathcal Y_n(\omega_j)\|^2$ diverges at least at rate $\psi_n^2\log n$. Since the estimated eigenvalues in the definition $T_n$ converge by assumption, the statistic $T_n$ must diverge.
\end{proof}

\subsection{Auxiliary lemmas}\label{sec:aux}
\subsubsection*{Auxiliary lemmas for standardising sequences}
\begin{lem} \label{munp}
Let $d=d_n\to\infty.$ Suppose $\lambda_1>\lambda_2$. Then we have that $b_n-b_n^d\to 0$ as $n\to\infty$, where $b_n=\lambda_1\log(q\prod_{j=2}^\infty(1-\lambda_j/\lambda_1)^{-1})$.
\end{lem}
\begin{proof}
We have that
\[
	b_n-b_n^d=
	\lambda_1\sum_{j=d+1}^\infty\log(1+\lambda_j/(\lambda_1-\lambda_j))
\]
and for any $j>1$
\[
	\log(1+\lambda_j/(\lambda_1-\lambda_j))
	\le\lambda_j/(\lambda_1-\lambda_j)
	\le\lambda_j/(\lambda_1-\lambda_2).
\]
The claim follows from $\sum_{j=1}^\infty\lambda_j<\infty$.
\end{proof}

\subsubsection*{Auxiliary lemmas for truncation}
\begin{lem}\label{lemma:trunc}
Suppose that $d\ge1$ is fixed and $\operatorname E\|X_t\|^r<\infty$ with $r>2$. Then
\[
	M_n-\tilde M_n:=\max_{1\leq k\leq q}\|\mathcal{X}_n^d(\omega_k)\|^2-\max_{1\leq k\leq q}\|\tilde{\mathcal{X}}_n^d(\omega_k)\|^2=o_P(1)\quad\text{as}\quad n\to\infty.
\]
\end{lem}
\begin{proof}
We have that $\cap_{t=1}^n\{X_t=\tilde X_t\}\subset\{M_n=\tilde M_n\}$. Hence,
\begin{align*}
	P(|M_n-\tilde M_n|>\varepsilon)
	&\le P(M_n\ne\tilde M_n)\\
	&\le P(\cup_{t=1}^n\{X_t\ne\tilde X_t\})\\
	&\le nP(\|X_1\|>n^{1/r})\to0	
\end{align*}
as $n\to\infty$ for each $\varepsilon>0$ since $X_t$'s have the same distribution and $\operatorname E\|X_1\|^r<\infty$. The proof is complete.
\end{proof}

\begin{lem}\label{lemma:convofcovop}
Suppose that  $\operatorname E\|X_1\|^r<\infty$ with some $r\ge2$. Then
\[
	\|\operatorname{Var}(X_1)-\operatorname{Var}(\tilde X_1)\|
	=o(n^{-(1-2/r)})
	\quad\text{as}\quad n\to\infty.
\]
\end{lem}
\begin{proof}
We have that
\[
\operatorname{Var}(\tilde X_1)
	=\operatorname E[I_{\{\|X_1\|\le n^{1/r}\}}(X_1\otimes X_1)]
	-\operatorname E[X_1I_{\{\|X_1\|\le n^{1/r}\}}]\otimes\operatorname E[X_1I_{\{\|X_1\|\le n^{1/r}\}}]
\]
and
\[
	\operatorname E[X_1I_{\{\|X_1\|\le n^{1/r}\}}]=-\operatorname E[X_1I_{\{\|X_1\|>n^{1/r}\}}]
\]
since $\operatorname EX_1=0$. Hence,
\begin{align*}
	&\|\operatorname{Var}(X_1)-\operatorname{Var}(\tilde X_1)\|=\\
	&=\|\operatorname E[(X_1\otimes X_1)I_{\{\|X_1\|>n^{1/r}\}}]+
	\operatorname E[X_1I_{\{\|X_1\|>n^{1/r}\}}]\otimes\operatorname E[X_1I_{\{\|X_1\|>n^{1/r}\}}]\|\\
	&\le2\operatorname E[\|X_1\|^2I_{\{\|X_1\|>n^{1/r}\}}]\\
	&\le2(\operatorname E[\|X_1\|^rI_{\{\|X_1\|>n^{1/r}\}}])^{2/r}\cdot n^{-(1-2/r)}.
\end{align*}
In the last step we used the H\"older inequality. The proof is complete.
\end{proof}

\begin{lem}\label{lem:convrateeigen}
Suppose that \autoref{ass:eigen} holds. Denote the eigenvectors of $\operatorname{Var}(\tilde X_1)$ by $\tilde v_1,\tilde v_2,\ldots$ with the corresponding eigenvalues $\tilde\lambda_1,\tilde\lambda_2,\ldots$ and $c_k=\operatorname{sgn}\langle v_k,\tilde v_k\rangle$ for $k\ge1$. Then
\[
	\|\tilde v_k-c_kv_k\|
	=o(n^{-(1-2/r)})
	\quad\text{as}\quad n\to\infty
\]
for each $k\ge1$.
\end{lem}
\begin{proof}
Using Lemma~2.3 of \citet{horvath2012},
\[
	\|\tilde v_k-c_kv_k\|
	\le\frac{2\sqrt 2}{\alpha_k}\|\operatorname{Var}(X_1)-\operatorname{Var}(\tilde X_1)\|,
\]
where $\alpha_1=\lambda_1-\lambda_2$ and $\alpha_k=\min\{\lambda_{k-1}-\lambda_k,\lambda_k-\lambda_{k+1}\}$ for $k>1$. We use \autoref{lemma:convofcovop} to conclude the proof.
\end{proof}

\begin{lem}\label{lem:higherordermoment}
Suppose that $\operatorname EX_1=0$ and $\operatorname E\|X_1\|^r<\infty$ with some $r\ge2$. Then for any $v>r$ we have
\[
	\operatorname E\|\tilde X_1\|^v
	=O(n^{v/r-1})\quad\text{as $n\to\infty$}
\]
\end{lem}
\begin{proof}
We have that
\begin{align*}
	\operatorname E\|\tilde X_1\|^v
	&\le\operatorname E\|X_1 I_{\{\|X_1\|\le n^{1/r}\}}-\operatorname E[X_1I_{\{\|X_1\|\le n^{1/r}\}}]\|^v\\
	&\le\operatorname E(\|X_1\| I_{\{\|X_1\|\le n^{1/r}\}}+\operatorname E[\|X_1\| I_{\{\|X_1\|\le n^{1/r}\}}])^v\\
	&\le2^v\operatorname E[\|X_1\|^v I_{\{\|X_1\|\le n^{1/r}\}}]\\
	&=2^v\operatorname E[\|X_1\|^r\|X_1\|^{v-r} I_{\{\|X_1\|\le n^{1/r}\}}]\\
	&=2^v\operatorname E \|X_1\|^r\cdot n^{v/r-1}.
\end{align*}
The proof is complete.
\end{proof}

\subsubsection*{Auxiliary lemma for CLT}
\begin{lem} \label{lemM1} Set  $V_t :=\xi_t^{d}\otimes f_t$ (for vectors $\otimes$ denotes the Kronecker product) with $1\le t\le n$ where $\xi_t^d$ and $f_t$ are given by \eqref{eq:xi_t} and \eqref{eq:Wt2q} respectively. Then
\begin{align} \label{Mbounds}
\frac{\lambda_d}{2}  \;  \leq n^{-1}\sum_{t=1}^n E|u' V_{t}|^2 \, \leq \frac{\lambda_1}{2}
\end{align}
for all $u\in \mathbb{S}^{2dq-1}$.
\end{lem} 

\begin{proof}[Proof of \autoref{lemM1}] Denote $u=(u_1',\dots,u_d')'\in\mathbb{R}^{2dq}$ with $u_k'\in\mathbb{R}^{2q}$ for $1\le k \le d$. Since
\[
	\operatorname E[\xi_t^{(d)}(\xi_t^{(d)})']
	=\operatorname{diag}(\lambda_1,\ldots,\lambda_d),
\]
we obtain
\begin{align*}
\frac{1}{n}\sum_{t=1}^n E|u' V_{t}|^2 &=  \frac{1}{n}\sum_{t=1}^n  \sum_{j,k=1}^d E[\langle X_t,v_j \rangle \langle X_t,v_k \rangle] u_j'f_t u_k'f_t     \\
&=  \sum_{j=1}^d \lambda_j u_j'\Bigl(\frac{1}{n}\sum_{t=1}^n f_t f_t' \Bigr)u_j.
 \end{align*}
But note that $n^{-1}\sum_{t=1}^n f_tf_t'=\frac{1}{2}I_{2q}$. Hence,
\begin{equation}\label{WW}
	\frac{1}{n}\sum_{t=1}^n E|u'V_{t}|^2 =  \frac{1}{2} \sum_{j=1}^d \lambda_j \|u_j \|^2
 \end{equation}
and \eqref{WW} is maximized if $\|u_1\|=1$ and minimized if $\|u_d\|=1$. 
\end{proof}

\section*{Acknowledgments}
Vaidotas Characiejus would like to acknowledge the support of the Communauté française de Belgique, Actions de Recherche Concertées, Projects Consolidation 2016-2021. The authors would like to thank Professor Kengo Kato for sharing a detailed proof of Nazarov's inequality (\citet{chernozhukov:2017b}) and Professor Fedor Nazarov who kindly communicated the proof of \autoref{thm:Ologn} in the univariate case on MathOverflow.

\newpage
\appendix

\section*{Appendix}

\section{Constants in the high-dimensional CLT}\label{bdep}
The constant $C$ in \eqref{cherrho} depends on the parameters $b$ and $s$ (in our setting, this corresponds to $\lambda_d$ and $2d$). When we let $d\to\infty$ then we need to make this dependence explicit. This is the purpose of our \autoref{p:cher+c}, which is an extension of Proposition~3.2 of \citet{chernozhukov:2017}. We outline here the modifications needed.
Proposition~3.2 of \citet{chernozhukov:2017} is based on a series of other results which we are now formulating in the adapted version.
An important step in this extension is the following lemma, which is a refinement of Lemma~A.1 of~\citet{chernozhukov:2017}. This result is originally due to~\citet{nazarov:2003}. For the proof, we refer to \citet{chernozhukov:2017b}.

\begin{lem} \label{lemNaz} Let $Y\sim N_p(0,\Sigma)$ be such that $EY_j^2 \geq b$ for all $j=1,\dots,p$ and with $b>0$. Then for every $y\in \mathbb{R}^p$ and $\delta>0$,
\begin{equation} \label{ineqNaz}
 P( Y\leq y+ \delta)-P(Y \leq y) \leq \frac{\delta}{b^{1/2}}(\sqrt{2\log p}+2),
\end{equation}
where the inequalities between vectors are coordinatewise.
\end{lem}

For the rest of this section we will use essentially the same notation as in \citet{chernozhukov:2017}, with exception of the constants $K_i$ for $i\geq 1$, which in our case are independent of the parameters $n$, $p$, $b$ and $s$. Moreover, we use $V_t$ and $W_t$ in \eqref{e:cbound} instead of  $X_t$  and $Y_t$, since the latter variables already have different usage in this paper. It will be assumed throughout that $p\geq 3$. Here is some notation needed later.
\begin{align*}
L_n&=\max_{1\leq j\leq p} n^{-1}\sum_{i=1}^n E|V_{i,j}|^3;\\
M_{n,V}(\phi)&=n^{-1}\sum_{i=1}^n E\Bigl(\max_{1\leq j\leq p} |V_{i,j}|^3I_{\{\max_{1\leq j\leq p} |V_{i,j}|>\sqrt{n}/(4\phi\log p)\}}\Bigr);\\
M_{n}(\phi)&=M_{n,V}(\phi)+M_{n,W}(\phi);\\
\end{align*}

\begin{lem}[Modification of Lemma 5.1 in \citet{chernozhukov:2017}]
Denote 
\begin{align*}
\varrho_n &= \sup_{y\in\mathbb{R}^p,v\in[0,1]}|P(\sqrt{v}S_n^V+\sqrt{1-v}S_n^W\leq y)-P(S_n^W\leq y)|.
\end{align*} 
Suppose that there exists some constant $b>0$ such that $n^{-1}\sum_{i=1}^n E[V_{i,j}^2]\geq b$ for all $j=1,\ldots, p$. Then for all $\phi\geq 1$ it holds that
$$
\varrho_n\leq  K_1 \Bigl\lbrace \frac{\phi^2\log^2p}{n^{1/2}}\Bigl( \phi L_n \varrho_n +  L_n \frac{\log^{1/2}p}{ b^{1/2}}  + \phi M_n(\phi) \Bigr)  +  \frac{\log^{1/2}p}{\phi \, b^{1/2}} \Bigr\rbrace.
$$
\end{lem}
\begin{proof}
Replace in the proof of Lemma 5.1 in \citet{chernozhukov:2017} the bound obtained  from their Lemma~A.1 by the Lemma~\ref{lemNaz}. This lemma is used in two places. At all other places the constant $K_1$ required in \citet{chernozhukov:2017} is not affected by the value of $b$.
\end{proof}

The lemma can be easily extended to hyperrectangles. Let $\mathcal{A}^{\text{re}}$ be the class of hyperrectangles in $\mathbb{R}^p$.

\begin{lem}[Modification of Corollary 5.1 in \citet{chernozhukov:2017}]\label{l:cor5.1}
Denote 
\begin{align*}
\varrho_n' &= \sup_{A\in\mathcal{A}^{\mathrm{re}},v\in[0,1]}|P(\sqrt{v}S_n^V+\sqrt{1-v}S_n^W\in A)-P(S_n^W\in A)|.
\end{align*} 
Suppose that there exists some constant $b>0$ such that $n^{-1}\sum_{i=1}^n E[V_{i,j}^2]\geq b$ for all $j=1,\ldots, p$. Then for all $\phi\geq 1$ it holds that
$$
\varrho_n'\leq  K_2 \Bigl\lbrace \frac{\phi^2\log^2p}{n^{1/2}}\Bigl( \phi L_n \varrho_n' +  L_n \frac{\log^{1/2}p}{ b^{1/2}}  + \phi M_n(2\phi) \Bigr)  +  \frac{\log^{1/2}p}{\phi \, b^{1/2}} \Bigr\rbrace.
$$
\end{lem}

\begin{lem}[Modification of Theorem 2.1 in \citet{chernozhukov:2017}]\label{l:th21}
Suppose that there exists some constant $b\in (0,1]$ such that $n^{-1}\sum_{i=1}^n E[V_{i,j}^2]\geq b$ for all $j=1,\ldots, p$. Then if $\overline{L}_n\geq L_n$, 
\begin{align*}
\rho_n(\mathcal{A}^{\mathrm{re}}) \leq  K_3 \Bigl\lbrace   \frac{ \log^{7/6}p \,\overline{L}_n^{1/3} }{b^{1/2}\, n^{1/6}}+
\frac{M_n(\phi_n)}{ \overline{L}_n} 
\Bigr\rbrace,
\end{align*}
where $\phi_n = \gamma\,\frac{n^{1/6}}{\overline{L}_n^{1/3} \log^{2/3} p}$ and $\gamma=\frac{1}{ K_2\vee 1 }$. 
\end{lem}

\begin{proof}
Note that $K_3 \frac{ \log^{7/6}p \,\overline{L}_n^{1/3} }{b^{1/2}\, n^{1/6}}=K_3(\gamma\,\frac{\log^{1/2}p}{b^{1/2}})\frac{1}{\phi_n}\geq \frac{K_3}{K_2\vee 1}\,\frac{1}{\phi_n} $. Thus the result is trivial if $\phi_n\leq 2$, because we can choose $K_3=2(K_2\vee 1)$. So we assume without loss of generality that $\phi_n\geq 2$ (hence $\phi\geq 1$) and apply Lemma~\ref{l:cor5.1} with $\phi=\phi_n/2$ given above. \end{proof}

In the following we refer to conditions (i)--(iv) from Section~\ref{proofsart3}.

\begin{lem}[Modification of Proposition 2.1 in \citet{chernozhukov:2017}]\label{p:pr21}
Suppose that there exists some constant $b\in (0,1]$ such that $n^{-1}\sum_{i=1}^n E[V_{i,j}^2]\geq b$ for all $j=1,\ldots, p$. Suppose, moreover, that condition (ii) holds for some sequence $B_n\geq 1$. Then, under (iv) we have
\begin{equation} \label{bounds2.1bis}
 \rho_n(\mathcal{A}^{\mathrm{re}}) \leq  K_4 \Bigl\{\frac{ B_n^{1/3}\; \log^{7/6}(pn)  }{b^{1/2}\; n^{1/6}}  +   \frac{ B_n^{2/3}\; \log(pn)  }{b^{1/2}\,n^{\frac{q-2}{3q}}}\Bigr\}.
\end{equation} 
\end{lem}

\begin{proof}
The proof is based on Lemma~\ref{l:th21}, choosing $\phi_n=\gamma (n^{-1} \overline{L}_n^2 \log^4 p)^{-1/6}$,  and $\overline{L}_n= B_n+\frac{B_n^2}{n^{1/2-2/q}\log^{1/2} p}$. In~\citet{chernozhukov:2017} exactly the same terms are used and worked out, but their bound corresponding to our \eqref{bounds2.1bis} does not involve the factor $b^{-1/2}$ (here it comes from our Lemma~\ref{l:th21}). The dependence on $b$ in their bound remains latent.  In particular it is implicit in the constant corresponding to our $\gamma$  (they denote it $K_2$). In our case this constant doesn't depend on $b$. 

We also note that \citet{chernozhukov:2017} request in their proof the constants
$$
\frac{ B_n^{1/3}\; \log^{7/6}(pn)  }{n^{1/6}}\leq \min\{C\gamma^{-1/2},\gamma/2\}\quad\text{and}\quad\frac{ B_n^{2/3}\; \log(pn)  }{n^{\frac{q-2}{3q}}}\leq \gamma/2,
$$
with some absolute constant $C$.  (See inequalities (32) and (33) in \citet{chernozhukov:2017}.)  We can impose these assumptions as well, since otherwise \eqref{bounds2.1bis} becomes trivial, by choosing $K_4=K_4(\gamma)$ big enough. Here we use again that our $\gamma$  doesn't depend on $b$. Hence, these assumptions will also not invoke dependence of $K_4$ on $b$.
\end{proof}

For the next result we need further terms and definitions, which one can find in Section~3 in~\citet{chernozhukov:2017}. For convenience we give here a quick review. For a closed convex set $A$ we define a mapping $\mathcal{S}_A$, which maps from  $v\in \mathbb{S}^{p-1}(=\{v\in \mathbb{R}^{p}\colon \|v\|=1\})$ to $\mathcal{S}_A(v)=\mathrm{sup}\{w'v\colon w\in A\}$. Then $A=\cap_{v\in \mathbb{S}^{p-1}}\{w\in\mathbb{R}^p\colon w'v\leq \mathcal{S}_A(v)\}$. If $A$ is a convex polytope with at most $m$ facets then it is called $m$-generated. If $\mathcal{V}(A)$ are the $m$ unit vectors orthogonal to the facets, then $A=\cap_{v\in \mathcal{V}(A)}\{w\in\mathbb{R}^p\colon w'v\leq \mathcal{S}_A(v)\}$. For an $m$-generated set $A^m$, set $A^{m,\epsilon}=\cap_{v\in \mathcal{V}(A^m)}\{w\in\mathbb{R}^p\colon w'v\leq \mathcal{S}_{A^m}(v)+\epsilon\}$. A convex set $A$ admits an approximation with precision $\epsilon$ by an $m$-generated convex set $A^m$ if $A^m\subset A\subset A^{m,\epsilon}$

We are now ready to define the class 
$\mathcal{A}^{\mathrm{si}}(d)$, which is the class  of Borel sets $A\in\mathbb{R}^p$ such that $A$ admits an approximation with precision $1/n$ by an $m$-generated convex set $A^m$ with $m\leq (pn)^d$. (In \citet{chernozhukov:2017} a more general class $\mathcal{A}^{\mathrm{si}}(d)$ is introduced, but for us only the case $a=1$ is relevant.) Consider $\mathcal{A}\subset \mathcal{A}^{\mathrm{si}}(d)$. For some $A\in\mathcal{A}$ let $A^m=A^m(A)$ be the approximating $m$-generated set. For the process $V_t$ let $\tilde V_t=(\tilde V_{t,1},\ldots, \tilde V_{t,m})'=(v'V_t)_{v\in \mathcal{V}(A^m)}$, $t=1,\ldots, n$ and consider the following conditions.

\begin{itemize}
\item[(i')] $n^{-1}\sum_{t=1}^n E|\tilde V_{t,j}|^2\ge b$  for all  $j=1,\ldots, m$;
\item[(ii')]$n^{-1}\sum_{t=1}^n E|\tilde V_{t,j}|^{2+k}\leq B_n^k$ for all $j=1,\ldots,m$ and $k=1,2$;
\item[(iv')]$E\max_{1\le j\le p}(|\tilde V_{t,j}|/B_n)^q\le 2$ for all $t=1,\ldots,n$.
\end{itemize}

\begin{lem}[Modification of Proposition 3.1 in \citet{chernozhukov:2017}]\label{p:pr31}
Let $\mathcal{A}$ be a subclass of $\mathcal{A}^{\mathrm{si}}(d)$ such that (i'), (ii') and (iv') are satisfied for all $A\in\mathcal{A}$.  Then
\begin{equation} \label{boundP31}
 \rho_n(\mathcal{A}) \leq  K_5 \Bigl\{\frac{ B_n^{1/3}\; \log^{7/6}\big((pn)^d\big)  }{b^{1/2}\; n^{1/6}}  +   \frac{ B_n^{2/3}\; \log\big((pn)^d\big)  }{b^{1/2}\,n^{\frac{q-2}{3q}}}\Bigr\}.
\end{equation} 
The constant $K_5$ does not depend on $d$.
\end{lem}

\begin{proof} Following the proof of Proposition~3.1 in \citet{chernozhukov:2017} and applying our Lemma~\ref{lemNaz} instead of their Lemma~A.1 we obtain
$$
|P(S_n^V\in A)-P(S_n^W\in A)|\leq \frac{1}{nb^{1/2}}(\sqrt{2\log \big((pn)^d\big)}+2)+\bar{\rho},
$$
where 
$$
\bar{\rho}=\max\big\{|P(S_n^V\in A^m)-P(S_n^W\in A^m)|,|P(S_n^V\in A^{m,\epsilon})-P(S_n^W\in A^{m,\epsilon})|\big\}.
$$
For $\bar{\rho}$ we can use Lemma~\ref{p:pr21}, and apply it to $\tilde V_1,\ldots, \tilde V_n$. From this we get the bound in \eqref{boundP31} which in turn dominates $\frac{1}{nb^{1/2}}(\sqrt{2\log \big((pn)^d\big)}+2)$.
\end{proof}

\begin{proof}[Proof of Proposition~\ref{p:cher+c}] We need to adapt the proof of Proposition~3.2 in \citet{chernozhukov:2017} and make the dependence on $b$ and $s$ explicit. Since we are interested in the case $b\to 0$, we can assume that $b\leq 1$. Here are the steps and modifications. In the following $C$ is an absolute constant which may vary from place to place.
\begin{enumerate}
\item It is sufficient to consider sparsely convex sets $A$ with $\max_{1\leq j\leq p}|w_j|\leq pn^{5/2}$ for all  $w=(w_1,\ldots, w_p)'\in A$. The argument is the same as in \citet{chernozhukov:2017}.
\item Consider the subclass $\mathcal{A}_1^\text{sp}(s)$ of sets in $\mathcal{A}^\text{sp}(s)$ which contain a ball of radius $1/n$. Using their Lemma~D.1 with $\gamma=1$ it is easy to show that $A\in \mathcal{A}_1^\text{sp}(s)$ is approximable by an $m$-generated set $A^m$ with precision $1/n$ and with $m\leq (pn)^{2 s^2}$, provided $n\geq n(\gamma)=n(1)$. This latter constraint is not a restriction, since  for $n<n(1)$ we may just choose a big enough constant $K_5$. The target is then to show conditions (i'), (ii') and (iv') and apply Lemma~\ref{p:pr31}. Condition (i') follows from condition (i) and the statement in Lemma~D.1 that $A^m$ can be chosen to satisfy $\|v\|_0\leq s$ for all $v\in \mathcal{V}(A^m)$. Next, in \citet{chernozhukov:2017}  it is shown that (ii') holds with $B_n$ replaced by $B_n'=B_ns^{3}$ and (iv'') with  $B_n$ replaced by $B_ns$. Since we require the original $B_n$ to be bounded we get from \eqref{boundP31}
\begin{equation}\label{e:finalbound}
\rho_n(\mathcal{A}_1^\text{sp}(s)) \leq  K_6 \Bigl\{\frac{ s^4\log^{7/6}\big(pn\big)  }{b^{1/2}\; n^{1/6}}\Bigr\}.
\end{equation}
\item Let $\mathcal{A}_2^\text{sp}(s)=\mathcal{A}^\text{sp}(s)\backslash \mathcal{A}_1^\text{sp}(s)$. Let us first consider the case of an $A\in \mathcal{A}_2^\text{sp}(s)$ where we have at least one $A_k$ in the representation $A=\cap A_k$ which does not contain a ball of radius $1/n$. Remember that $I_{A_k}(x)$ depends only on $s$ components of $x\in\mathbb{R}^p$, say $\tilde x=(x_{j_1},\ldots, x_{j_s})\in\mathbb{R}^s$. Define a convex set $\widetilde{A}_k\subset \mathbb{R}^s$ such that $I_{A_k}(x)=I_{\widetilde{A}_k}(\tilde x)$ for all $x\in\mathbb{R}^p$. For $J=J(A_k)=(j_1,\ldots, j_s)$ we then have $\{S_n^V\in A_k\}=\{S_{n,J}^V\in \tilde A_k\}$. By Lemma~A.2 in \citet{chernozhukov:2017} it follows that
\begin{align*}
P(S_n^W\in A)\leq P(S_n^W\in A_k)=P(S_{n,J}^W\in \widetilde{A}_k)\leq C\frac{1}{n}\sqrt{\|\Omega^{-1}_J\|_\mathcal{S}}\leq C\frac{1}{n^{3/2}}\frac{s^{1/4}}{\sqrt{b}},
\end{align*}
where $\Omega_J=\mathrm{Var}(S_{n,J}^W)$. For the second inequality above we use that $\|\Omega^{-1}_J\|^2_\mathcal{S}\leq \frac{s}{\lambda^2_\text{min}}$, where $\lambda_\text{min}$ is the smallest eigenvalue of $\Omega_J$ and hence $1/\lambda_\text{min}$ is the largest eigenvalue of $\Omega^{-1}_J$. By our condition (i) we have $\lambda_\text{min}\geq n b$. Next we bound
\begin{align*}
&|P(S_n^V\in A_k)-P(S_n^W\in A_k)|=|P(S_{n,J}^V\in \widetilde A_k)-P(S_{n,J}^W\in \widetilde A_k)|\\
&\quad\leq \sup_{M\in\mathcal{C}}|P(\Omega_J^{-1/2}S_{n,J}^V\in M)-P(N_s(0,I_s)\in M)|=:\Delta,
\end{align*}
where $\mathcal{C}$ is the class of measurable convex sets in $\mathbb{R}^s$.  In \citet{goetze:1991} 
it is shown that for some absolute constant $C$ we have
$$
\Delta\leq C s\beta_3 ,\quad\text{where}\quad \beta_3=\sum_{t=1}^nE\|\Omega_J^{-1/2}V_{t,J}\|^3.
$$
Note that by (ii)
$$
 \beta_3= \|\Omega_J^{-1/2}\|^3 \sum_{t=1}^n E\|V_{t,J}\|^3\leq \frac{s^{3/4} B_n}{b^{3/2}\,n^{1/2}}.
$$
We can assume that $b^{1/2}n^{1/6}\geq 1$, otherwise the bound in Proposition~\ref{p:cher+c} becomes trivial, by choosing the constant big enough. Then $b^{1/2}n^{1/6}\leq b^{3/2}n^{1/2}$ and therefore
\begin{align*}
P(S_n^V\in A)&\leq P(S_n^V\in A_k)\\
&\leq P(S_n^W\in A_k)+|P(S_n^V\in A_k)-P(S_n^W\in A_k)|\\
&\leq C\frac{s^{7/4}}{b^{3/2}n^{1/2}}B_n.
\end{align*}
This shows that both, $P(S_n^V\in A)$ and $P(S_n^W\in A)$, are dominated by $\frac{s^4\log^{7/6}(pn)}{b^{1/2}n^{1/6}}$ and hence this is also true for the difference $|P(S_n^V\in A)-P(S_n^W\in A)|$.
\item The last case we need to handle is when $A\in\mathcal{A}^{\text{sp}}_2$ and $A=\cap_{k=1}^K A_k$, such that each $A_k$ contains a ball with radius $1/n$. We show that both, $P(S_n^V\in A)$ and $P(S_n^W\in A)$ are dominated by the bound in \eqref{e:finalbound}. Thus their difference is.

Like in Step~2 we can find for each $k$ an $m$-generated convex set $A^m$ such that $A_k^m\subset A_k\subset A_k^{m,1/n}$. We have  $m\leq (pn)^{2 s^2}$ and we can choose $A^m$ such that for all $v\in\mathcal{V}(A_k^m)$ we have $\|v\|_0\leq s$. In \citet{chernozhukov:2017} it is shown that quite generally $K\leq p^s$. Thus, 
$A_0:=\cap_{k=1}^K A_k^{m,1/n}$ is approximable by an $m'$-generated set with $m'\leq p^s(pn)^{2 s^2}\leq(pn)^{3 s^2}$. Using the same arguments as in Step~2 we see that (i'), (ii') and (iv') hold and hence by Lemma~\ref{p:pr31} with $d=3s^2$ we have that  $|P(S_n^V\in A_0)-P(S_n^W\in A_0)|$ is bounded as in~\eqref{e:finalbound}. Now since $A$ contains no ball of radius $1/n$ we get $P(S_n^W\in \cap_{k=1}^K A_k^{m,-1/n})=0$ and hence
\begin{align*}
P(S_n^W\in A)&\leq P(S_n^W\in A_0)\\
&= P(S_n^W\in \cap_{k=1}^K A_k^{m,1/n})-P(S_n^W\in \cap_{k=1}^K A_k^{m,-1/n})\\
&\leq P(v'S_n^W\leq \mathcal{S}_{A_k^m}(v)+1/n\colon k=1,\ldots, K, v\in\mathcal{V}(A_k^m))\\
&\qquad - P(v'S_n^W\leq \mathcal{S}_{A_k^m}(v)-1/n\colon k=1,\ldots, K, v\in\mathcal{V}(A_k^m))\\
&\leq \frac{2}{n}\frac{\sqrt{\log((pn)^{3s^2})}+2}{\sqrt{b}}.
\end{align*}
For the last inequality we used Lemma~\ref{lemNaz}.
Finally we observe
\[
P(S_n^V\in A)\leq P(S_n^V\in A_0)\leq P(S_n^W\in A_0)+|P(S_n^V\in A_0)-P(S_n^W\in A_0)|.
\]
\end{enumerate}
The proof is complete.
\end{proof}

\section{Maximum of linear forms}
Suppose that $X_1,\ldots,X_n$ are iid zero mean random elements with values in $H$ and  that  $\{a_{jnt}\}_{1\le j\le q,1\le t\le n}\subset \mathcal L(H)$ are such that $\|a_{jnt}\|\le n^{-1/2}$ for $n\ge1$. Denote
\[
	L_{nj}
	=\sum_{t=1}^na_{jnt}(X_t)
\]
for $n\ge1$ and $1\le j\le q$. We show that $\operatorname E\max_{1\le j\le q}\|L_{nj}\|^2=O(\log n)$ as $n\to\infty$ provided that $\operatorname E\|X_1\|^r<\infty$ with some $r>2$. We first prove an auxiliary lemma that is used in the proof.

\begin{lem}\label{lemma:refinedbound}
Suppose that $X_1,\ldots,X_n$ are independent zero mean random elements with values in $H$ such that $\|X_t\|\le b$ a.s.\ with some $b>0$ and $p\coloneqq P(\|X_t\|\ne0)$ for each $t=1,\ldots,n$. Then
\[
	P(\|S_n\|\ge x)
	\le2e^{-\frac{x^2}{b^2}\beta}\bigl[1+p(e^{\frac{x^2}{2b^2}\beta^2}-1)\bigr]^n
\]
for $x\ge 0$ and each $\beta\in\mathbb R$, where $S_n=X_1+\ldots+X_n$ for $n\ge1$.
\end{lem}
\begin{proof}
It follows from Theorem~3.5 of \citet{pinelis:1994} that $P(\|S_n\|\ge x)\le 2\exp\{-x^2/(2nb^2)\}$ for all $x\ge0$. Let $A_k$ denote the event that $k$ out of $n$ random elements $X_1,\ldots,X_n$ are not equal to $0$ with $k=0,\ldots,n$. Using the fact that $k^{-1}\ge2\beta-\beta^2k$ for $k\ne0$ and $\beta\in\mathbb R$, we obtain
\begin{align*}
	P(\|S_n\|\ge x)
	&=\sum_{k=0}^nP(\|S_n\|\ge x\mid A_k)P(A_k)\\
	&\le2\sum_{k=1}^ne^{-\frac{x^2}{2kb^2}}{n\choose k}p^k(1-p)^{n-k}\\
	&\le2e^{-\frac{x^2}{b^2}\beta}\sum_{k=1}^n{n\choose k}e^{\frac{x^2}{2b^2}\beta^2k}p^k(1-p)^{n-k}\\
	&\le2e^{-\frac{x^2}{b^2}\beta}\bigl[1+p(e^{\frac{x^2}{2b^2}\beta^2}-1)\bigr]^n
\end{align*}
for $x>0$ and all $\beta\in\mathbb R$. The proof is complete.
\end{proof}
Now we are ready to prove the main result.
\begin{Theorem}\label{thm:Ologn}
Suppose that $X_1,\ldots,X_n$ are iid random elements with values in $H$ such that $\operatorname EX_1=0$ and $\operatorname E\|X_1\|^r$ with some $r>2$. Then
\[
	\operatorname E\max_{1\le j\le q}\|L_{nj}\|^2
	=O(\log n)\quad\text{as}\quad n\to\infty.
\]
\end{Theorem}
The proof of \autoref{thm:Ologn} is based on a decomposition of random elements. Consider some zero mean random element $\xi$ with values in $H$ such that $\operatorname E\|\xi\|^r=1$. Suppose that $\{p_k\}_{k\ge0}$ is a strictly decreasing sequence of probabilities that converges to $0$ as $k\to\infty$. Choose $R\geq 1$. The decomposition of $\xi$ is given by
\begin{equation}\label{eq:decomp}
	\xi
	=\hat\xi_0+\sum_{k=1}^R\check{\xi}_k+\xi_R^\prime,
\end{equation}
where the random elements are defined in the following way. Since we are only interested in the distribution, we can assume without loss of generality that a uniform random variable can be defined on the underlying probability space. Then the space is non-atomic and hence there exists an event $F_R$ such that $P(F_R)=p_R$ and $\|\xi\|\le p_R^{-1/r}$ on $F^c_R$. Define
\[
	\hat\xi_R
	=\xi I_{F_R^c}+p_R^{-1}\operatorname E[\xi I_{F_R}]I_{F_R}
	\quad\text{and}\quad
	\xi_R^\prime
	=\xi-\hat\xi_R.
\]
The remaining random elements are defined recursively. Denote $\hat F_{k-1}$ the event such that  $P(\hat F_{k-1})=p_{k-1}$ and $\|\hat\xi_k\|\le p_{k-1}^{-1/r}$ on $\hat F_{k-1}^c$ with $1\leq k\leq R$. Moreover, define
\[
	\hat\xi_{k-1}
	=\hat\xi_k I_{\hat F_{k-1}^c}+p_{k-1}^{-1}\operatorname E[\hat\xi_k I_{\hat F_{k-1}}]I_{\hat F_{k-1}}
	\quad\text{and}\quad
	\check\xi_k
	=\hat\xi_k-\hat\xi_{k-1}.
\]
Then decomposition \eqref{eq:decomp} has the following properties
\begin{enumerate}[(i)]
\item$\operatorname E\hat\xi_0=\operatorname E\check\xi_1=\ldots=\operatorname E\check\xi_R=\operatorname E\xi_R'=0$;
\item$\operatorname E\|\hat\xi_0\|^r\le\operatorname E\|\hat\xi_1\|^r\le\ldots\le\operatorname E\|\hat\xi_{R}\|^r\le\operatorname E\|\xi\|^r$;
\item$\operatorname E\|\xi_R'\|^r\le2^r\operatorname E\|\xi\|^r$ and $\operatorname E\|\xi_k\|^r\le2^r\operatorname E\|\tilde\xi_k\|^r$ for $1\le k\le R$;
\item by H\"older's inequality, $\|p_R^{-1}\operatorname E[\xi I_{F_R}]\|\le p_R^{-1/r}$ and $\|p_{k-1}^{-1}\operatorname E[\hat\xi_k I_{\hat F_{r-1}}]\|\le p_{k-1}^{-1/r}$ for $1\le k\le R$ and hence $\|\hat\xi_k\|\le p_k^{-1/r}$ for $0\le k\le R$ and $\|\check\xi_k\|\le2p_{k-1}^{-1/r}\le2p_k^{-1/r}$ for $1\le k\le R$;
\item$P(\xi_R'\ne0)\le p_R$ and $P(\check\xi_k\ne0)\le p_{k-1}$ for $1\le k\le R$.
\end{enumerate}

\begin{proof}[Proof of \autoref{thm:Ologn}]
Assume that $\operatorname E\|X_1\|^r=1$ without loss of generality. We use decomposition~\eqref{eq:decomp} with $p_k=2^{-k}$ and $R=\log_2n$ (the logarithm to the base $2$). Then
\[
	L_{nj}
	=\sum_{t=1}^na_{jnt}(\hat X_{t,0})
	+\sum_{k=1}^{\log_2n}\sum_{t=1}^na_{jnt}(\check X_{t,k})
	+\sum_{t=1}^na_{jnt}(X_{t,\log_2n}')
\]
for $1\le j\le q$ and $n\ge1$. Observe that $\hat X_{t,0}=0$ almost surely for $1\le t\le n$ since $p_0=1$.

By H\"older's inequality,
\[
	n^{-1/2}\sum_{t=1}^n\|X_{t,\log_2n}'\|
	\le n^{-1/2}N_n^{1-1/r}\Bigl[\sum_{t=1}^n\|X_{t,\log_2n}'\|^r\Bigr]^{1/r},
\]
where the random variable $N_n=\sum_{t=1}^nI_{\{\|X_{t,\log_2n}'\|\ne 0\}}$ follows a binomial distribution. Observe that 
\[
	\operatorname Ee^{N_n}
	=[1+P(\|X_{1,\log_2n}'\|\ne0)(e-1)]^n
	\le[1+n^{-1}(e-1)]^n
	<e^e.
\]
It follows that any fixed moment of $N_n$ is bounded for all $n\ge1$ and hence
\[
	\operatorname E\max_{1\le j\le q}\Bigl\|\sum_{t=1}^na_{jnt}(X_{t,\log_2n}')\Bigr\|^2
	\le n^{-1}\operatorname EN_n^{2-2/q}[n\operatorname E\|X_{1,\log_2n}'\|^r]^{2/r}
	=O(n^{2/r-1})
\]
as $n\to\infty$ using Jensen's inequality.

By the triangle inequality,
\begin{align*}
	\operatorname E\max_{1\le j\le q}\Bigl\|\sum_{k=1}^{\log_2n}\sum_{t=1}^na_{jnt}(\check X_{t,k})\Bigr\|^2
	&\le\operatorname E\Bigl[\sum_{k=1}^{\log_2n}\max_{1\le j\le q}\Bigl\|\sum_{t=1}^na_{jnt}(\check X_{t,k})\Bigr\|\Bigr]^2\\
	&\le\Bigl|\sum_{k=1}^{\log_2n}\Bigl(\operatorname E\max_{1\le j\le q}\Bigl\|\sum_{t=1}^na_{jnt}(\check X_{t,k})\Bigr\|^2\Bigr)^{1/2}\Bigr|^2.
\end{align*}
Choose $\delta>0$ such that $1/r+\delta<1/2$. 
We show that
\[
	\operatorname E\max_{1\le j\le q}\Bigl\|\sum_{t=1}^n2^{\delta k-1}a_{jnt}(\check X_{t,k})\Bigr\|^2
	\le C^2\log n+O(1)
\]
for each $1\le k\le\log_2n$ with some $C>0$. Note that this yields the proof. More specifically, we show that
\begin{align}
	&\operatorname E\Bigl[\Bigl\|\sum_{t=1}^n2^{\delta k-1}a_{jnt}(\check X_{t,k})\Bigr\|^2-C^2\log n\Bigr]_+\nonumber\\
	&\quad=\int_{C\sqrt{\log n}}^\infty2xP\Bigl(\Bigl\|\sum_{t=1}^n2^{\delta k-1}a_{jnt}(\check X_{t,k})\Bigr\|>x\Bigr)dx
	\le\frac Cn \label{eq:maxmoment}
\end{align}
for each $1\le j\le q$ and $1\le k\le\log_2n$.

Since $\|2^{\delta k-1}a_{jnt}(\check X_{t,k})\|\leq p_k^{-(1/r+\delta)}n^{-1/2}$ and $P(\check X_{t,k}\ne0)\le p_{k-1}$ for $1\le t\le n$, we are in the position to apply \autoref{lemma:refinedbound}. Let $s_k\coloneqq p_k^{-(1/r+\delta)}$ and let $\beta=\beta'/n$. We thus obtain
\begin{equation}\label{eq:refinedbound}
	P\Bigl(\Bigl\|\sum_{t=1}^n2^{\delta k-1}a_{jnt}(\check X_{t,k})\Bigr\|>x\Bigr)
	\le2e^{-\frac{x^2}{s_k^2}\beta'}\bigl[1+p_{k-1}(e^{\frac{x^2}{2ns_k^2}(\beta')^2}-1)\bigr]^n.
\end{equation}
We split the integral in equation~\eqref{eq:maxmoment} into two integrals
\begin{multline*}
	\int_{C\sqrt{\log n}}^\infty2xP\Bigl(\Bigl\|\sum_{t=1}^n2^{\delta k-1}a_{jnt}(\check X_{t,k})\Bigr\|>x\Bigr)dx
	=\int_{C\sqrt{\log n}}^{\sqrt n/s_k}2xP\Bigl(\Bigl\|\sum_{t=1}^n2^{\delta k-1}a_{jnt}(\check X_{t,k})\Bigr\|>x\Bigr)dx\\
	+\int_{\sqrt n/s_k}^\infty2xP\Bigl(\Bigl\|\sum_{t=1}^n2^{\delta k-1}a_{jnt}(\check X_{t,k})\Bigr\|>x\Bigr)dx.
\end{multline*}
Using~\eqref{eq:refinedbound} and $p_{k-1} \leq 2s_k^{-2}$, setting $\beta'=\varepsilon s_k^2$, where $\varepsilon>0$ is a small number, and using the inequality $(e^x-1)\le 2x$, which holds for small enough values of $x$, we obtain
\begin{align*}
	\int_{C\sqrt{\log n}}^{\sqrt n/s_k}2xP\Bigl(\Bigl\|\sum_{t=1}^n2^{\delta k-1}a_{jnt}(\check X_{t,k})\Bigr\|>x\Bigr)dx
	&\le2\int_{C\sqrt{\log n}}^{\sqrt n/s_k}2xe^{-x^2\varepsilon}\bigl[1+\frac{2x^2\varepsilon^2}n\bigr]^ndx\\
	&\le2\int_{C\sqrt{\log n}}^{\sqrt n/s_k}2xe^{-x^2\varepsilon(1-2\varepsilon)}dx\\
	&=\frac2{\varepsilon(1-2\varepsilon)}\bigl[e^{-C^2\varepsilon(1-2\varepsilon)\log n}-e^{-\frac{n}{s_k^2}\cdot\varepsilon(1-2\varepsilon)}\bigr]\\
	&\le\frac2{\varepsilon(1-2\varepsilon)}\cdot n^{-C^2\varepsilon(1-2\varepsilon)},
\end{align*}
where $C$ is chosen in such a way that $C^2\varepsilon(1-2\varepsilon)\ge1$.

Denote $\rho=1/2-(1/r+\delta)>0$. Using~\eqref{eq:refinedbound}, setting $\beta'=\varepsilon s_k\sqrt n/x$, where $\varepsilon>0$ is a small number, again using the inequality $(e^x-1)\le 2x$, which holds for small enough values of $x$, as well as the inequality $(1+x/n)^n\le e^x$ for $x\in\mathbb R$ and $n\ge1$, we obtain
\begin{align*}
	\int_{\sqrt n/s_k}^\infty2xP\Bigl(\Bigl\|\sum_{t=1}^n2^{\delta k-1}a_{jnt}(\check X_{t,k})\Bigr\|>x\Bigr)dx
	&\le\int_{\sqrt n/s_k}^\infty2x2e^{-\frac{x\sqrt n}{s_k}\varepsilon}\bigl[1+p_{k-1}(e^{\frac{\varepsilon^2}{2}}-1)\bigr]^ndx\\
	&\le4[1+p_{k-1}\varepsilon^2]^n\int_{\sqrt n/s_k}^\infty xe^{-\frac{x\sqrt n}{s_k}\varepsilon}dx\\
	&\le4\varepsilon^{-1}e^{\frac{2n\varepsilon^2}{s_k^2}}e^{-\frac n{s_k^2}\varepsilon}\bigl[1+\frac{s_k^2}{\varepsilon n}\bigr]\\
	&\le4\varepsilon^{-1}e^{-\varepsilon(1-2\varepsilon)n^{1-2(1/r+\delta)}}\bigl[1+\varepsilon^{-1}n^{2(1/r+\delta)-1}\bigr]\\
	&=4\varepsilon^{-1}e^{-\varepsilon(1-2\varepsilon)n^{2\rho}}\bigl[1+\varepsilon^{-1}n^{-2\rho}\bigr],
\end{align*}
where we used the fact that $p_{k-1}\le 2s_k^{-2}$ and $s_k\leq n^{1/r+\delta}$ for $1\le k\le\log_2n$. The proof is complete.
\end{proof}

\newpage
\bibliographystyle{plainnat}
\bibliography{references}

\end{document}